\documentclass[10pt]{extarticle}
\usepackage{amsmath}
\usepackage[dvips]{graphicx}
\usepackage{textcomp}
\usepackage{amsbsy}
\usepackage{latexsym}
\usepackage{mathtools}
\usepackage[mathscr]{eucal}
\usepackage{amsfonts}
\usepackage{amssymb}
\usepackage[usenames]{color}
\usepackage{stmaryrd}

\usepackage{cases}

\usepackage{algorithm,algorithmic}

\usepackage{epsf}

\usepackage{stackrel}

\usepackage{comment}

\def\RR{\rm \hbox{I\kern-.2em\hbox{R}}}
\def\NN{\rm \hbox{I\kern-.2em\hbox{N}}}
\def\ZZ{\rm {{\rm Z}\kern-.28em{\rm Z}}}
\def\CC{\rm \hbox{C\kern -.5em {\raise .32ex \hbox{$\scriptscriptstyle
|$}}\kern
-.22em{\raise .6ex \hbox{$\scriptscriptstyle |$}}\kern .4em}}

\def\<{\langle}
\def\>{\rangle}

\newcommand{\N}{{\mathbb N}}

\def\Chi{\raise .3ex
\hbox{\large $\chi$}} 
\def\lsima{\hbox{\kern -.6em\raisebox{-1ex}{$~\stackrel{\textstyle<}{\sim}~$}}\kern -.4em}
\def\lsim{\hbox{\kern -.2em\raisebox{-1ex}{$~\stackrel{\textstyle<}{\sim}~$}}\kern -.2em}
\def\[{\Bigl [}
\def\]{\Bigr ]}
\def\({\Bigl (}
\def\){\Bigr )}
\def\[{\Bigl [}
\def\]{\Bigr ]}
\def\({\Bigl (}
\def\){\Bigr )}

\makeatletter
\@addtoreset{equation}{section}
\makeatother
\newcommand{\be}{\begin{equation}}
\newcommand{\ee}{\end{equation}}
\newcommand{\beu}{\begin{equation*}}
\newcommand{\eeu}{\end{equation*}}
\newcommand{\bea}{$$ \begin{array}{lll}}
\newcommand{\eea}{\end{array} $$}
\newcommand{\bi}{\begin{itemize}}
\newcommand{\ei}{\end{itemize}}

\newtheorem{theorem}{Theorem}[section]
\newtheorem{remark}{Remark}[section]
\newtheorem{lemma}{Lemma}[section]

\newtheorem{corollary}{Corollary}[section]

\newtheorem{proof}{Proof}[section]
\newtheorem{definition}{Definition}[section]

\DeclareMathOperator*{\argmax}{argmax}
\DeclareMathOperator*{\argmin}{argmin}

\begin{document}
\bibliographystyle{plain}
\title
{
Multivariate approximation in downward closed polynomial spaces
\footnote{
This research is supported by Institut Universitaire de France and the ERC AdV project BREAD.}
}
\author{
Albert Cohen\thanks{
Sorbonne Universit\'es, UPMC Univ Paris 06, CNRS, UMR 7598, Laboratoire Jacques-Louis Lions, 4, place Jussieu 75005, Paris, France. 
email: cohen@ljll.math.upmc.fr
}
\
and 
\
Giovanni Migliorati\thanks{
Sorbonne Universit\'es, UPMC Univ Paris 06, CNRS, UMR 7598, Laboratoire Jacques-Louis Lions, 4, place Jussieu 75005, Paris, France. 
email: migliorati@ljll.math.upmc.fr
}
}

\date{December 15, 2016}
\maketitle

\begin{abstract}
\noindent
The task of approximating a function of $d$ variables from its evaluations at a given number of points is ubiquitous in numerical analysis and engineering applications. When $d$ is large, this task is challenged by the so-called {\em curse of dimensionality}. As a typical example, standard polynomial spaces, such as those of total degree type, are often uneffective to reach a prescribed accuracy unless a prohibitive number of evaluations is invested. In recent years it has been shown that, for certain relevant applications, there are substantial advantages in using certain {\it sparse} polynomial spaces having anisotropic features with respect to the different variables. These applications include in particular the numerical approximation of high-dimensional parametric and stochastic partial differential equations. We start by surveying several results in this direction, with an emphasis on the numerical algorithms that are available for the construction of the approximation, in particular through interpolation or discrete least-squares fitting. All such algorithms rely on the assumption that the set of multi-indices associated with the polynomial space is {\it downward closed}. 
In the present paper we introduce some tools for the study of approximation in multivariate spaces under this assumption, and use them in the derivation of error bounds, sometimes independent of the dimension $d$, and in the development of adaptive strategies.
\end{abstract}

{\bf AMS classification numbers:}   
41A65, 
41A25, 
41A10, 
41A05, 
65N12, 
65N15.  


\section{Introduction}\label{sec:1}

The mathematical modeling of complex physical phenomena often demands for functions that depend on a large number of variables. One typical instance occurs when a quantity of interest  $u$ is given as the solution to an equation written in general form as
\begin{equation}
{\cal P}(u,y)=0,
\label{eq:problem}
\end{equation}
where $y=(y_j)_{j=1,\dots,d}\in  \mathbb{R}^d$ is a vector that concatenates various physical parameters which have an influence on $u$.

Supposing that we are able to solve the above problem, either exactly or approximately by numerical methods, for any $y$ in some domain of interest $U\subset\mathbb{R}^d$ we thus have access to the {\it parameter-to-solution} map 
\begin{equation}
y \mapsto u(y).
\label{eq:solmap}
\end{equation}
The quantity $u(y)$ may be of various forms, namely:
\begin{enumerate}
\item
a real number, that is, $u(y)\in\mathbb{R}$;
\item
a function in some Banach space, for example when~\eqref{eq:problem} is a partial differential equation (PDE);
\item
a vector of eventually large dimension, in particular when~\eqref{eq:problem} is a PDE whose solution is numerically approximated using some numerical method with fixed discretization parameters. 
\end{enumerate}

As a guiding example which will be further discussed in this paper, consider the elliptic diffusion equation
\begin{equation}
-{\rm div}(a\nabla u)=f,
\label{eq:ellip}
\end{equation}
set on a given bounded Lipschitz domain $D\subset \mathbb{R}^k$ (say with $k=1,2$ or $3$), for some fixed right-hand side $f\in L^2(D)$, homogeneous Dirichlet boundary conditions $u_{|\partial D}=0$, and where $a$ has the general form
\begin{equation}
a=a(y)=\overline a+\sum_{j\geq 1} y_j\psi_j.
\label{eq:para}
\end{equation}
Here, $\overline a$ and $\psi_j$ are given functions in $L^\infty(D)$, and the $y_j$ range in finite intervals that,
up to renormalization, can all be assumed to be $[-1,1]$. In this example $y=(y_j)_{j\geq 1}$ is countably infinite-dimensional, that is, $d=\infty$. 
The standard weak formulation of~\eqref{eq:ellip} in $H^1_0(D)$,
$$
\int_D a\nabla u\nabla v=\int_D fv, \quad v\in H^1_0(D),
$$
is ensured to be well-posed for all such $a$ under the so-called {\em uniform ellipticity assumption}
\begin{equation}
\sum_{j\geq 1}|\psi_j| \leq \overline a-r, \quad \mbox{a.e. on } D,
\label{eq:uea}
\end{equation}
for some $r>0$. In this case, the map $y\mapsto u(y)$ acts from $U=[-1,1]^\mathbb{N}$ to $H^1_0(D)$.
However, if we consider the discretization of~\eqref{eq:ellip} in some finite element space $V_h  \subset H_0^1(D)$,
where $h$ refers to the corresponding mesh size, using for instance the Galerkin method, then the resulting map 
\begin{equation*}
y\mapsto u_h(y),
\end{equation*} 
acts from $U=[-1,1]^\mathbb{N}$ to $ V_h$. Likewise, if we consider a quantity of interest such as
the flux $q(u)=\int_{\Sigma} a \nabla u \cdot \sigma$ over a given interface ${\Sigma} \subset D$ 
with $\sigma$ being the outward pointing normal vector,
then the resulting map 
$$
y\mapsto q(y)=q(u(y)),
$$
acts from $U=[-1,1]^\mathbb{N}$ to $\mathbb{R}$. In all three cases, the above maps act from $U$ to some finite- or infinite-dimensional Banach space $V$, which is either $H^1_0$, $V_h$ or $\mathbb{R}$.

In the previous instances, the functional dependence between the input parameters $y$ and the output $u(y)$ is described in clear mathematical terms by equation~\eqref{eq:problem}. In other practical instances, the output $u(y)$ can be the outcome of a complex physical experiment or numerical simulation with input parameter $y$. However the dependence on $y$ might not be given in such clear mathematical terms. 

In all the abovementioned cases, we assume that are we are able to query the map~\eqref{eq:solmap} at any given parameter value $y\in U$, eventually up to some uncertainty. Such uncertainty may arise due to:
\begin{itemize}
\item[(i)] measurement errors, when $y\mapsto u(y)$ is obtained by a physical experiment, or 
\item[(ii)] computational errors, when $y\mapsto u(y)$ is obtained by a numerical computation.
\end{itemize}
The second type of errors may result from the spatial discretization when
solving a PDE with a given method, and from the round-off errors
when solving the associated discrete systems.

One common way of modeling such errors is by assuming that we
observe $u$ at the selected points $y$ up to a an additive noise
$\eta$ which may depend on $y$, that is, we evaluate
\begin{equation}
y\mapsto u(y)+\eta(y),
\label{eq:noisy}
\end{equation}
where $\eta$ satisfies a uniform bound
\begin{equation}
\|\eta\|_{L^\infty(U,V)}:=\sup_{y\in U} \|\eta(y)\|_V\leq \varepsilon,
\label{eq:noisebound}
\end{equation} 
for some $\varepsilon>0$ representing the noise level.

Queries of the exact $u(y)$ or of the noisy $u(y)+\eta(y)$ are often expensive since they require numerically solving a PDE, or setting up a physical experiment,
or running a time-consuming simulation algorithm. A natural objective is therefore to approximate the map~\eqref{eq:solmap} from some fixed number $m$ of such queries at points $\{y^1,\dots,y^m\}\in U$.  Such approximations $y\mapsto \widetilde u(y)$ are sometimes called {\it surrogate or reduced models}. 

Let us note that approximation of the map~\eqref{eq:solmap} is sometimes a preliminary task for solving other eventually more complicated problems, such as:
\begin{enumerate}
\item 
{\bf Optimization and Control}, 
\emph{i.e.}~find a $y$ which minimizes a certain criterion  
depending on $u(y)$. 
In many situations, the criterion takes the form of a convex functional of 
$u(y)$, and the minimization is subject to feasibility constraints. See \emph{e.g.}~the monographs \cite{BV,NW2006} and references therein for an overview of classical formulations and numerical methods for optimization problems. 

\item
{\bf Inverse Problems}, 
\emph{i.e.}~find an estimate $y$ from some data depending on the output $u(y)$. 
Typically, we face an ill-posed problem,  
where the parameter-to-solution map does not admit a global and stable inverse. 
Nonetheless, developing efficient numerical methods for approximating the parameter-to-solution map, \emph{i.e.}~solving the so-called direct problem, is a first step towards the construction of numerical methods for solving the more complex inverse problem, 
see \emph{e.g.}~\cite{Tara}. 

\item
{\bf Uncertainty Quantification}, \emph{i.e.}~describe the stochastic properties 
of the solution $u(y)$ in the case where the parameter $y$ is modeled by a random
variable distributed according to a given probability density. We may for instance
be interested in computing the expectation or variance of the $V$-valued random variable $u(y)$.
Note that this task amounts in computing multivariate integrals over the domain $U$
with respect to the given probability measure.  This area also embraces, among others, optimization and inverse problems whenever affected by uncertainty in the data. We refer to \emph{e.g.}~\cite{LK} for the application of polynomial approximation to uncertainty quantification, and to \cite{Stu} for the Bayesian approach to inverse problems.  
\end{enumerate}

There exist many approaches for approximating an unknown function of one or several variables from its evaluations at given points. One of the most classical approach consists in picking the approximant in a given suitable $n$-dimensional space of elementary functions,
such that
\begin{equation}
n\leq m.
\label{eq:mn}
\end{equation}
Here, by ``suitable'' we mean that the space should have the ability to approximate 
the target function to some prescribed accuracy, taking for instance advantage of its 
smoothness properties. By ``elementary'' we mean that such functions should
have simple explicit form which can be efficiently exploited in numerical computations. 
The simplest type of such functions are obviously polynomials in the variables $y_j$.
As a classical example, we may decide to use, for some given $k\in\mathbb{N}_0$, the total degree polynomial space of order $k$, namely
\begin{equation*}
\mathbb{P}_k:={\rm span}\left\{y\mapsto y^\nu\; : \; |\nu|:=\|\nu\|_1=\sum_{j=1}^d \nu_j\leq k\right\}, 
\end{equation*}
with the standard notation
\begin{equation*}
\quad y^\nu:=\prod_{j=1}^dy_j^{\nu_j}, \quad \nu=(\nu_j)_{j=1,\dots,d}.
\end{equation*}
Note that since $u(y)$ is $V$-valued, this means that we actually use the $V$-valued polynomial space
\begin{equation*}
\mathbb{V}_k:=V\otimes \mathbb{P}_k=\left\{y\mapsto \sum_{|\nu|\leq k} w_\nu y^\nu\; : \; w_\nu\in V\right\}.
\end{equation*}
Another classical example is the polynomial space of degree $k$ in each variable, namely
\begin{equation*}
\mathbb{Q}_k:={\rm span}\left\{y\mapsto y^\nu\; : \; \|\nu\|_\infty=\max_{j=1,\dots,d} \nu_j\leq k\right\}. 
\end{equation*}
A critical issue encountered by choosing such spaces is the fact that, for a fixed value of  $k$, the dimension of $\mathbb{P}_k$ grows with $d$ like $d^k$, and that of $\mathbb{Q}_k$ grows like $k^d$, that is, exponentially in $d$. Since capturing the fine structure of the map~\eqref{eq:solmap} typically requires a large polynomial degree $k$ in some coordinates, we expect in view of~\eqref{eq:mn} that the number of needed evaluations $m$ becomes prohibitive as the number of variables becomes large. 
This state of affairs is a manifestation of the so-called {\it curse of dimensionality}. From an approximation theoretic or information-based complexity point of view, the curse of dimensionality is expressed by the fact that functions in standard smoothness classes such as $C^s(U)$ cannot be approximated in $L^\infty(U)$ with better rate then $n^{-s/d}$ by any method using $n$ degrees of freedom or $n$ evaluations, see \emph{e.g.}~\cite{DHM,NW}.

Therefore, in high dimension, one is enforced to give up on classical polynomial spaces of the above form, and instead consider more general spaces of the general form
\begin{equation}
\mathbb{P}_\Lambda:={\rm span}\{y\mapsto y^\nu\; : \; \in \Lambda\},
\label{eq:plambda}
\end{equation}
where $\Lambda$ is a subset of $\mathbb{N}^d_0$ with a given cardinality $n:=\#(\Lambda)$. In the case of infinitely many variables $d=\infty$, 
we replace $\mathbb{N}_0^d$ by the set
$$
{\cal F}:=\ell^0(\mathbb{N},\mathbb{N}_0):=\{\nu=(\nu_j)_{j\geq 1} \; :\: \#({\rm supp}(\nu))<\infty\}, 
$$
of finitely supported sequences of nonnegative integers. For $V$-valued functions, we thus use the space
$$
\mathbb{V}_\Lambda=V\otimes \mathbb{P}_\Lambda:=\left\{y\mapsto \sum_{\nu\in\Lambda} w_\nu y^\nu\; : \; w_\nu\in V\right\}.
$$
Note that $\mathbb{V}_\Lambda=\mathbb{P}_\Lambda$ in the particular case where $V=\mathbb{R}$.

The main objective when approximating the map~\eqref{eq:solmap} is to maintain a reasonable trade-off between accuracy measured in a given error norm and complexity measured by $n$, exploiting the
different importance of each variable. Intuitively, large polynomial degrees should only be allocated to the most important variables. In this sense, if $d$ is the dimension and $k$ is the largest polynomial degree in any variable appearing in $\Lambda$, we view $\Lambda$ as a very {\it sparse} subset of $\{0,\dots,k\}^d$.

As generally defined by~\eqref{eq:plambda}, the space $\mathbb{P}_\Lambda$ does not satisfy some natural properties of usual polynomial spaces such as closure under differentiation in any variable, or invariance by a change of basis when replacing the monomials $y^\nu$ by other tensorized basis functions of the form
$$
\phi_\nu(y)=\prod_{j\geq 1} \phi_{\nu_j}(y_j),
$$
where the univariate functions $\{\phi_0,\dots,\phi_k\}$ form a basis of $\mathbb{P}_k$ for any $k\geq 0$, for example with the Legendre or Chebyshev polynomials. In order to fulfill these requirements, we ask that the set $\Lambda$ has the following natural property.

\begin{definition}
\label{thm:defdc}
A set $\Lambda\subset \mathbb{N}_0^d$ or $\Lambda\subset {\cal F}$ is {\em downward closed} if and only if
 \begin{equation*}
\nu \in \Lambda \textrm{ and } \widetilde \nu \leq \nu  \implies \widetilde \nu \in \Lambda,
\end{equation*}
where $ \widetilde \nu \leq \nu$ means that $\widetilde \nu_j\leq \nu_j$ for all $j$.
\end{definition}

Downward closed sets are also called lower sets. We sometimes use the terminology of downward closed polynomial spaces for the corresponding $\mathbb{P}_\Lambda$. To our knowledge, such spaces have been first considered in \cite{K} in the bivariate case $d=2$ and referred to as {\em polyn\^omes pleins}. Their study in general dimension $d$ has been pursued in \cite{LL} and \cite{DR}.
The objective of the present paper is to give a survey of recent advances on the use of downward closed polynomial spaces for high-dimensional approximation.

The outline is the following. We review in Section~\ref{sec:2} several polynomial approximation results obtained in \cite{BCM1,BCDM}
in which the use of well-chosen index sets allows one to {\em break the curse of dimensionality} for relevant classes of functions
defined on $U=[-1,1]^d$, \emph{e.g.}~such as those occuring when solving the elliptic PDE~\eqref{eq:ellip}
with parametric diffusion coefficient~\eqref{eq:para}. Indeed, we obtain an algebraic convergence rate $n^{-s}$, where
$s$ is independent of $d$ in the sense that such a rate may even hold when $d=\infty$. Here, we consider the error between the map~\eqref{eq:solmap} and its approximant in either norms $L^\infty(U,V)=L^\infty(U,V,d\mu)$ or $L^2(U,V)=L^2(U,V,d\mu)$, where $d\mu$ is the uniform probability measure,
$$
d\mu:=\bigotimes_{j\geq 1} \frac {dy_j} 2.
$$
We also consider the case of lognormal diffusion coefficients of the form 
\begin{equation}
a=\exp(b), \quad b=b(y)=\sum_{j\geq 1} y_j\psi_j,
\label{eq:lognormal}
\end{equation}
where the $y_j$ are i.i.d. standard Gaussians random variables. 
In this case, we have $U=\mathbb{R}^d$ and the error is measured in $L^2(U,V,d\gamma)$ where
\begin{equation}
d\gamma:=\bigotimes_{j\geq 1} g(y_j)dy_j, \quad g(t):=\frac 1 {\sqrt{2\pi}} e^{-t^2/2},
\label{eq:gaussmeasure}
\end{equation}
is the tensorized Gaussian probability measure. The above approximation results are established by using $n$-term truncations of polynomial expansions, such as Taylor, Legendre or Hermite, which do not necessarily result in downward closed index sets. In the present paper we provide a general approach to establish similar convergence rates with downward closed polynomial spaces.

The coefficients in the polynomial expansions cannot be computed exactly from a finite number of point evaluations of~\eqref{eq:solmap}.
One first numerical procedure that builds a polynomial approximation from point evaluations is interpolation. In this case the number $m$ of 
samples is exactly equal to the dimension $n$ of the polynomial space. We discuss in Section~\ref{sec:3} a general strategy to choose evaluation points and compute the interpolant in arbitrarily high dimension. One of its useful features is that the evaluations and interpolants are updated in a sequential manner as the polynomial space is enriched. We study the stability of this process and its ability to achieve the same convergence rates in $L^\infty$ established in Section~\ref{sec:2}. 

A second numerical procedure for building a polynomial approximation is the least-squares method, which applies to the overdetermined case $m>n$.  
To keep the presentation concise, we confine to results obtained in the analysis of this method only for the case of evaluations at random points. 
In Section~\ref{sec:4} we discuss standard least squares, both in the noisy and noiseless cases, and in particular explain under which circumstances the method is stable and compares favorably with the best approximation error in $L^2$. Afterwards we discuss the more general method of weighted least squares, which allows one to optimize the relation between the dimension of the polynomial space $n$ and the number of evaluation points $m$ that warrants stability and optimal accuracy.

The success of interpolation and least squares is critically tied to the choice of proper downward closed sets $(\Lambda_n)_{n\geq 1}$ with $\#(\Lambda_n)=n$. 
The set $\Lambda_n^*$ that gives the best polynomial approximation among all possible downward closed sets of cardinality $n$ is often not 
accessible. 
In practice we need to rely on some a-priori analysis to select ``suboptimal yet good'' sets. An alternative strategy is to select the sequence 
$(\Lambda_n)_{n\geq 1}$ in an adaptive manner, that is, make use of the computation of the approximation for $\Lambda_{n-1}$ in order to choose $\Lambda_n$. We discuss in Section~\ref{sec:5} several adaptive and nonadaptive strategies which make use of the downward closed structure of such sets. While our paper is presented in the framework of polynomial approximation, the concept of downward closed set may serve to define multivariate approximation procedures in other nonpolynomial frameworks. At the end of the paper we give some remarks on this possible extension, including, as a particular example, approximation by sparse piecewise polynomial spaces using hierarchical bases, such as sparse grid spaces. 

Let us finally mention that another class of frequently used methods in high-dimensional approximation is based on Reproducing Kernel Hilbert Space (RKHS) or equivalently Gaussian process regression, also known as {\it kriging}. In such methods, for a given Mercer kernel $K(\cdot,\cdot)$ the approximant is typically searched by minimizing the associated RKHS norm among all functions agreeing with the data at the evaluation points,
or equivalently by computing the expectation of a Gaussian process with covariance function $K$ conditional to the observed data. Albeit natural competitors, these methods do not fall in the category discussed in the present paper, in the sense that the space where the approximation is picked varies with the evaluation points. It is not clear under which circumstances they may also break the curse of dimensionality.

\section{Sparse Approximation in Downward Closed Polynomial Spaces}
\label{sec:2}

\subsection{Truncated Polynomial Expansions}
\label{sec:2.1}

As outlined in the previous section, we are interested in deriving polynomial approximations of the map~\eqref{eq:solmap} acting from $U=[-1,1]^d$ with $d\in \mathbb{N}$ or $d=\infty$ to the Banach space $V$. Our first vehicle to derive such approximations, together with precise error bounds for relevant classes of maps, consists in truncating certain polynomial expansions of~\eqref{eq:solmap} written in general form as
\begin{equation}
\sum_{\nu\in{\cal F}} u_\nu \phi_\nu,
\label{eq:expan}
\end{equation}
where for each $\nu=(\nu_j)_{j\geq 1}\in{\cal F}$ the function $\phi_\nu:U\to \mathbb{R}$ has the tensor product form
\begin{equation*}
\phi_\nu(y)=\prod_{j\geq 1} \phi_{\nu_j}(y_j),
\end{equation*}
and $u_\nu\in V$. Here we assume that $(\phi_k)_{k\geq 0}$ is a sequence of univariate polynomials such that $\phi_0\equiv 1$ and the degree of $\phi_k$ is equal to $k$. This implies that $\{\phi_0,\dots,\phi_k\}$ is a basis of $\mathbb{P}_k$ and that the above product only involves a finite number of factors, even in the case where $d=\infty$. Thus, we obtain polynomial approximations of~\eqref{eq:solmap} by fixing some sets $\Lambda_n\subset {\cal F}$ with $\#(\Lambda_n)=n$ and defining
\begin{equation}
u_{\Lambda_n}:=\sum_{\nu\in\Lambda_n} u_\nu \phi_\nu.
\label{eq:nterm}
\end{equation}
Before discussing specific examples, let us make some general remarks on the truncation of countable expansions with $V$-valued coefficients, not necessarily of tensor product or polynomial type.

\begin{definition}
The series~\eqref{eq:expan} is said to converge {\it conditionally} with limit $u$ in a given norm $\|\cdot\|$ if and
only if there exists an exhaustion $(\Lambda_n)_{n\geq 1}$ of ${\cal F}$ 
(which means that for any $\nu\in{\cal F}$ there exists $n_0$ such that $\nu\in\Lambda_n$ for all $n\geq n_0$),
with the convergence property
\begin{equation}
\lim_{n\to +\infty} \left\|u-u_{\Lambda_n}\right\| =0.
\label{eq:condconv}
\end{equation}
The series~\eqref{eq:expan} is said to converge {\it unconditionally} towards $u$ in the same norm, if and only if~\eqref{eq:condconv} holds for every exhaustion $(\Lambda_n)_{n\geq 1}$ of ${\cal F}$. 
\end{definition}

As already mentioned in the introduction, we confine our attention to the error norms  $L^\infty(U,V)$ or $L^2(U,V)$ with respect to the uniform probability measure $d\mu$. We are interested in establishing unconditional convergence, as well as estimates of the error between $u$ and its truncated expansion, for both types of norm.

In the case of the $L^2$ norm, unconditional convergence can be established when $(\phi_\nu)_{\nu\in{\cal F}}$ is an orthonormal basis of $L^2(U)$.
In this case we know from standard Hilbert space theory that if~\eqref{eq:solmap} belongs to $L^2(U,V)$ then the inner products
$$
u_\nu:=\int_{U} u(y)\phi_\nu(y) \, d\mu,\quad \nu\in{\cal F},
$$
are elements of $V$, and the series~\eqref{eq:expan} converges unconditionally towards $u$ in $L^2(U,V)$. In addition, the error is given by
\begin{equation}
\left\|u-u_{\Lambda_n}\right\|_{L^2(U,V)}= \left( \sum_{\nu\notin\Lambda_n} \|u_\nu\|_V^2\right)^{1/2},
\label{eq:estimL2}
\end{equation}
for any exhaustion $(\Lambda_n)_{n\geq 1}$. Let us observe that, since $d\mu$ is a probability measure, the $L^\infty(U,V)$ norm controls the $L^2(U,V)$ norm, and thus the above holds whenever the map $u$ is uniformly bounded over $U$.

For the $L^\infty$ norms, consider an expansion~\eqref{eq:expan} where the functions $\phi_\nu:U\mapsto \mathbb{R}$ are normalized such that $\|\phi_\nu\|_{L^\infty(U)}=1$, for all $\nu\in{\cal F}$. 
Then $(\|u_\nu\|_V)_{\nu\in {\cal F}}\in \ell^1({\cal F})$, and it is easily checked that, whenever the expansion~\eqref{eq:expan} converges conditionally to a function $u$ in $L^\infty(U,V)$, it also converges unconditionally to $u$ in $L^\infty(U,V)$. In addition, for any exhaustion $(\Lambda_n)_{n\geq 1}$, we have the error estimate
\begin{equation}
\left\|u-u_{\Lambda_n}\right\|_{L^\infty(U,V)}\leq \sum_{\nu\notin\Lambda_n} \|u_\nu\|_V.
\label{eq:estimLinf}
\end{equation}
The above estimate is simply obtained by triangle inequality, and therefore generally it is not as sharp as~\eqref{eq:estimL2}.
One particular situation is when $(\phi_\nu)_{\nu\in {\cal F}}$ is an orthogonal basis of $L^2(U)$ normalized in $L^\infty$.
Then, if  $u\in L^2(U,V)$ and the  
$$
u_\nu:=\frac {1}{\|\phi_\nu\|_{L^2(U,V)}^2}\int_{U} u(y)\phi_\nu(y) \, d\mu,\quad \nu\in{\cal F},
$$
satisfy  $(\|u_\nu\|_V)_{\nu\in {\cal F}}\in \ell^1({\cal F})$, we find on the one hand that~\eqref{eq:expan} converges unconditionally to a limit in
$L^\infty(U,V)$ and in turn in $L^2(U,V)$. On the other hand, we know that it converges toward $u\in L^2(U,V)$. Therefore, its limit in $L^\infty(U,V)$ is also $u$. 

A crucial issue is the choice of the sets $\Lambda_n$ that we decide to use when defining the $n$-term truncation~\eqref{eq:nterm}. 
Ideally, we would like to use the set $\Lambda_n$ which minimizes the truncation error in some given norm $\|\cdot\|$ among all sets of cardinality $n$. 

In the case of the $L^2$ error, if $(\phi_\nu)_{\nu\in{\cal F}}$ is an orthonormal basis of $L^2(U)$, the estimate~\eqref{eq:estimL2} shows that the optimal $\Lambda_n$ is the set of indices corresponding to the $n$ largest $\|u_\nu\|_V$. This set is not necessarily unique, in which case any realization of $\Lambda_n$ is optimal.

In the case of the $L^\infty$ error, there is generally no simple description of the optimal $\Lambda_n$. However, when the $\phi_\nu$ are normalized in $L^\infty(U)$, the right-hand side in the estimate~\eqref{eq:estimLinf} provides an upper bound for the truncation error.
This bound is minimized by again taking for $\Lambda_n$ the set of indices corresponding to the $n$ largest $\|u_\nu\|_V$, with the error now bounded by the $\ell^1$ tail of the sequence $(\|u_\nu\|_V)_{\nu\in {\cal F}}$, in contrast to the $\ell^2$ tail which appears in~\eqref{eq:estimL2}. 

The properties of a given sequence $(c_\nu)_{\nu\in{\cal F}}$ which ensure a certain rate of decay $n^{-s}$ of its $\ell^q$ tail after one retains its $n$ largest entries are well understood.
Here, we use the following result, see \cite{CD}, originally invoked by Stechkin in the particular case $q=2$. This result says that the rate of decay is governed by the $\ell^p$ summability of the sequence for values of $p$ smaller than $q$.

\begin{lemma}
\label{thm:stechkin}
Let $0<p<q<\infty$ and let $(c_\nu)_{\nu\in{\cal F}}\in \ell^p({\cal F})$ be a sequence of nonnegative numbers.
Then, if $\Lambda_n$ is a set of indices which corresponds to the $n$ largest $c_\nu$, one has
$$
\left(\sum_{\nu\notin\Lambda_n} c_\nu^q\right)^{1/q} \leq C (n+1)^{-s}, \quad C:=\|(c_\nu)_{\nu\in{\cal F}}\|_{\ell^p}, \quad s:=\frac 1 q-\frac 1 p.
$$
\end{lemma}

In view of~\eqref{eq:estimL2} or~\eqref{eq:estimLinf}, application of the above result 
shows that $\ell^p$ summability of the sequence $(\|u_\nu\|_V)_{\nu\in {\cal F}}$
implies a convergence rate $n^{-s}$ when retaining  the terms 
corresponding to the $n$ largest $\|u_\nu\|_V$ in~\eqref{eq:expan}.   
From~\eqref{eq:estimL2}, when $(\phi_\nu)_{\nu \in{\cal F}}$ is an
orthonormal basis, we obtain $s=\frac 1 p-\frac 1 2$ if $p<2$.
From~\eqref{eq:estimLinf}, when the $\phi_\nu$ are normalized in $L^\infty(U)$,
we obtain $s=\frac 1 p-1$ if $p<1$.

In the present setting of polynomial approximation, we mainly consider four types of series corresponding to four different choices of the univariate functions $\phi_k$:
\begin{itemize}
\item
Taylor (or power) series of the form
\begin{equation}
\sum_{\nu\in{\cal F}} t_\nu y^\nu, \quad t_{\nu}:=\frac 1 {\nu !}\partial^\nu u(y=0), \quad \nu!:=\prod_{j\geq 1} \nu_j!,
\label{eq:taylor}
\end{equation}
with the convention that $0!=1$.
\item
Legendre series of the form
\begin{equation}
\sum_{\nu\in{\cal F}} w_\nu L_\nu(y), \quad L_\nu(y)=\prod_{j\geq 1}L_{\nu_j}(y_j),\quad w_\nu:=\int_U u(y)L_\nu(y) \, d\mu,
\label{eq:legendre}
\end{equation}
where $(L_k)_{k\geq 0}$ is the sequence of Legendre polynomials on $[-1,1]$ 
normalized with respect to the uniform measure $\int_{-1}^1 |L_k(t)|^2 \frac {dt}2=1$,
so that $(L_\nu)_{\nu\in {\cal F}}$ is an orthonormal basis of $L^2(U,d\mu)$.
\item
Renormalized Legendre series of the form
\begin{equation}
\sum_{\nu\in{\cal F}} \widetilde w_\nu \widetilde L_\nu(y), \quad \widetilde L_\nu(y)=\prod_{j\geq 1}\widetilde L_{\nu_j}(y_j), \quad \widetilde w_\nu:=\left(\prod_{j\geq 1} (1+2\nu_j)\right)^{1/2} w_\nu,
\label{eq:renormlegendre}
\end{equation}
where $(\widetilde L_k)_{k\geq 0}$ is the sequence of Legendre polynomials on $[-1,1]$ with the standard normalization 
$\|\widetilde L_k\|_{L^\infty([-1,1])}=\widetilde L_k(1)=1$, so that $\widetilde L_k=(1+2k)^{-1/2} L_k$.
\item
Hermite series of the form
\begin{equation}
\sum_{\nu\in{\cal F}} h_\nu H_\nu(y), \quad H_\nu(y)=\prod_{j\geq 1}H_{\nu_j}(y_j),\quad h_\nu:=\int_U u(y)H_\nu(y) \, d\gamma,
\label{eq:hermite}
\end{equation}
with $(H_k)_{k\geq 0}$ being the sequence of Hermite polynomials normalized according to
$\int_{\mathbb{R}} | H_k(t)|^2 g(t)dt=1$, and $d\gamma$ given by~\eqref{eq:gaussmeasure}. In this case $U=\mathbb{R}^d$ and
$(H_\nu)_{\nu \in {\cal F}}$ is an orthonormal basis of $L^2(U,d\gamma)$.
\end{itemize}

We may therefore estimate the $L^2$ error resulting from the truncation of the Legendre series~\eqref{eq:legendre}
by application of~\eqref{eq:estimL2}, or the $L^\infty$ error resulting from the truncation of the Taylor series~\eqref{eq:taylor}
or renormalized Legendre series~\eqref{eq:renormlegendre} by application of~\eqref{eq:estimLinf}. According to Lemma~\ref{thm:stechkin},
we derive convergence rates that depend on the value of $p$ such that of the coefficient sequences 
$(\|t_\nu\|_V)_{\nu\in{\cal F}}$,  $(\|w_\nu\|_V)_{\nu\in{\cal F}}$, $(\|\widetilde w_\nu\|_V)_{\nu\in{\cal F}}$ or $(\|h_\nu\|_V)_{\nu\in{\cal F}}$  
belong to $\ell^p({\cal F})$. 

In a series of recent papers such summability results have
been obtained for various types of parametric PDEs.
We refer in particular to \cite{CDS, BCM1} for the elliptic PDE~\eqref{eq:ellip}
with affine parameter dependence~\eqref{eq:para}, to \cite{HS,BCDM} for
the lognormal dependence~\eqref{eq:lognormal}, and to \cite{CCS2} for more
general PDEs and parameter dependence. 
One specific feature is that these conditions can 
be fulfilled in the infinite-dimensional framework.
We thus obtain convergence rates that are immune to the 
curse of dimensionality, in the sense that they hold with $d=\infty$.
Here, we mainly discuss the results established in \cite{BCM1,BCDM} which have the specificity
of taking into account the support properties of the functions $\psi_j$.

One problem with this approach is that the sets $\Lambda_n$ associated to the $n$ largest
values in these sequences are generally not downward closed. 
In the next sections, we revisit these results in order to establish similar convergence rates for approximation in downward closed polynomial spaces.

\subsection{Summability Results}
\label{sec:2.2}

The summability results in \cite{BCM1,BCDM} are based on  
certain weighted $\ell^2$ estimates which can be established for the previously defined
coefficient sequences under various relevant conditions for the elliptic PDE~\eqref{eq:ellip}. 
We first report below these weighted estimates. The first one from \cite{BCM1} concerns the affine
parametrization~\eqref{eq:para}. Here, we have $V=H^1_0(D)$ and $V'$ denotes its dual $H^{-1}(D)$.

\begin{theorem}
\label{thm:theosumaff}
Assume that $\rho=(\rho_j)_{j\geq 1}$ is a sequence of positive numbers such that 
\begin{equation}
\sum_{j\geq 1} \rho_j|\psi_j(x)|\leq \overline a(x)-\widetilde r, \quad x\in D,
\label{eq:uearho}
\end{equation} 
for some fixed number $\widetilde r >0$.
Then, one has 
\begin{equation}
\sum_{\nu\in {\cal F}}\left( \rho^\nu \|t_\nu\|_V\right)^2<\infty, \quad \rho^\nu=\prod_{j\geq 1} \rho_j^{\nu_j},
\label{eq:wl2t}
\end{equation}
as well as 
\begin{equation}
\sum_{\nu\in {\cal F}}\left(\beta(\nu)^{-1} \rho^\nu \|w_\nu\|_V\right)^2=\sum_{\nu\in {\cal F}}\left(\beta(\nu)^{-2} \rho^\nu \|\widetilde w_\nu\|_V\right)^2<\infty, 
\label{eq:wl2l}
\end{equation}
with 
$$
\beta(\nu):=\prod_{j\geq 1} (1+2\nu_j)^{1/2}.
$$
The constants bounding these sums depend on $\widetilde r$, $\|f\|_{V'}$, $\overline a_{\min}$ and $\|\overline a\|_{L^\infty}$.
\end{theorem}

A few words are in order concerning the proof of these estimates. The first estimate~\eqref{eq:wl2t} is established by
first proving that the uniform ellipticity assumption~\eqref{eq:uea} implies the $\ell^2$ summability 
of the Taylor sequence $(\|t_\nu\|_V)_{\nu\in{\cal F}}$. Since the assumption~\eqref{eq:uearho} 
means that~\eqref{eq:uea} holds with the $\psi_j$ replaced by $\rho_j\psi_j$, this gives the $\ell^2$ summability of the Taylor sequence for the renormalized map 
$$
y\mapsto u(\rho y), \quad \rho y=(\rho_jy_j)_{j\geq 1},
$$
which is equivalent to~\eqref{eq:wl2t}. The second estimate is established by first showing that $\sum_{j\geq 1} \rho_j|\psi_j|\leq \overline a-\widetilde r$ implies finiteness of the weighted Sobolev-type norm
$$
\sum_{\nu\in{\cal F}}\frac{ \rho^{2\nu} }{\nu!}\int_U  \left\| \partial^\nu u(y) \right\|_V^2 \prod_{j\geq 1} (1 - |y_j|)^{2\nu_j} \, d\mu<\infty.
$$
Then, one uses the Rodrigues formula $L_k(t) = \left(\frac {d}{dt}\right)^k\left(  \frac{\sqrt{2 k + 1}}{k! \,2^{k}} ( t^2 - 1)^{k} \right)$ in each variable $y_j$ to bound the weighted $\ell^2$ sum in~\eqref{eq:wl2l} by this norm.

\begin{remark}
As shown in \cite{BCM1}, the above result remains valid for more general classes of orthogonal polynomials of Jacobi type, such as the Chebyshev polynomials which are associated with the univariate measure $\frac {dt}{2\pi\sqrt{1-t^2}}$. 
\end{remark}

The second weighted $\ell^2$ estimate from \cite{BCDM} concerns the lognormal parametrization~\eqref{eq:lognormal}.

\begin{theorem} Let $r\geq 0$ be an integer.
Assume that there exists a positive sequence $\rho=(\rho_j)_{j\geq 1}$ such that 
$\sum_{j\geq 1} \exp(-\rho_j^2)<\infty$ and such that 
\begin{equation}
\sum_{j\geq 1} \rho_j|\psi_j(x)|=K< C_r:=\frac{\ln 2}{\sqrt r}, \quad x\in D.
\label{eq:psicr}
\end{equation}
Then, one has 
\begin{equation}
\label{eq:wl2h}
\sum_{\nu\in {\cal F}} \xi_\nu \|h_\nu\|_V^2<\infty,
\end{equation}
where
$$
\xi_\nu:=\sum_{\|\widetilde \nu\|_{\ell^\infty}\leq r}{\nu\choose \widetilde \nu} \rho^{2\widetilde \nu}=\prod_{j\geq 1}\left(\sum_{l=0}^r {\nu_j\choose l}\rho_j^{2l}\right),
\quad {\nu\choose \widetilde \nu}:=\prod_{j\geq 1}{\nu_j\choose \widetilde \nu_j},
$$
with the convention that ${k \choose l}=0$ when $l>k$.
The constant bounding this sum depends on $\|f\|_{V'}$, $\sum_{j\geq 1} \exp(-\rho_j^2)$ and on the difference $C_r-K$.
\end{theorem}

Similar to the weighted $\ell^2$ estimate~\eqref{eq:wl2l}  for the Legendre coefficients, the proof of~\eqref{eq:wl2h}
follows by first establishing finiteness of a weighed Sobolev-type norm
$$
\sum_{\|\nu\|_{\ell^\infty} \leq r}\frac {\rho^{2\nu}}{\nu !} \int_U \|\partial^\nu u(y)\|_V^2 \, d\gamma<\infty,
$$
under the assumption~\eqref{eq:psicr} in the above theorem. Then one uses the Rodrigues formula
$H_k(t) =  \frac{(-1)^k}{\sqrt{k!} } \frac{g^{(k)}(t)}{g(t)}$, with $g$ given by~\eqref{eq:gaussmeasure}, in each variable $y_j$
to bound the weighted $\ell^2$ sum in~\eqref{eq:wl2h} by this norm.

In summary, the various estimates expressed in the above theorems all take the form
$$
\sum_{\nu\in{\cal F}} (\omega_\nu c_\nu)^2 <\infty,
$$
where 
$$
c_\nu\in \{ \|t_\nu\|_V, \|w_\nu\|_V,\|\widetilde w_\nu\|_V,\|h_\nu\|_V\},
$$
or equivalently
$$
\omega_\nu\in \{\rho^\nu,\rho^\nu \beta(\nu)^{-1},\rho^\nu \beta(\nu)^{-2}, \xi_\nu^{1/2}\}.
$$

Then, one natural strategy for establishing $\ell^p$ summabilty of the sequence $(c_\nu)_{\nu\in{\cal F}}$ is to
invoke H\"older's inequality, which gives, for all $0<p<2$,
$$
\left(\sum_{\nu\in{\cal F}} |c_\nu|^p\right)^{1/p}\leq \left(\sum_{\nu\in{\cal F}} (\omega_\nu c_\nu)^2\right)^{1/2} \left(\sum_{\nu\in{\cal F}} |\kappa_\nu|^q\right)^{1/q} <\infty, \quad \frac 1 q:=\frac 1 p-\frac 1 2,
$$
where the sequence $(\kappa_\nu)_{\nu\in {\cal F}}$ is defined by
$$
\kappa_\nu:=\omega_\nu^{-1}.
$$
Therefore $\ell^p$ summability of $(c_\nu)_{\nu\in{\cal F}}$ follows from $\ell^q$ summability of 
$(\kappa_\nu)_{\nu\in{\cal F}}$ with $0<q<\infty$ such that $\frac 1 q=\frac 1 p-\frac 1 2$.
This $\ell^q$ summability can be related to that of the univariate sequence 
$$
b=(b_j)_{j\geq 1},\quad b_j:=\rho_j^{-1}.
$$
Indeed, from the factorization
$$
\sum_{\nu\in{\cal F}} b^{q\nu}=\prod_{j\geq 1} \sum_{n\geq 0} b_j^{nq},
$$
one readily obtains the following elementary result, see \cite{CD} for more details.

\begin{lemma} 
\label{thm:lemmaq}
For any $0<q<\infty$, one has
$$
b\in \ell^q(\mathbb{N})\quad {\rm and}\quad \|b\|_{\ell^\infty}<1 \iff (b^\nu)_{\nu\in{\cal F}} \in \ell^q({\cal F}).
$$
\end{lemma}

In the case $\omega_\nu=\rho^\nu$, \emph{i.e.}~$\kappa_\nu=b^\nu$, this shows that the 
$\ell^p$ summability of the Taylor coefficients $(\|t_\nu\|_V)_{\nu\in{\cal F}}$
follows if the assumption~\eqref{eq:uearho} holds with 
$b=(\rho_j^{-1})_{j\geq 1}\in \ell^q(\mathbb{N})$ and $\rho_j>1$ for all $j$.
By a similar factorization, it is also easily checked that for any algebraic factor of the form $\alpha(\nu):=\prod_{j\geq 1}(1+c_1\nu_j)^{c_2}$ with $c_1,c_2\geq 0$, one has
$$
b\in \ell^q(\mathbb{N})\quad {\rm and}\quad \|b\|_{\ell^\infty}<1 \iff (\alpha(\nu)b^\nu)_{\nu\in{\cal F}} \in \ell^q({\cal F}).
$$
This allows us to reach a similar conclusion in the cases $\omega_\nu=\beta(\nu)^{-1}\rho^\nu$ or $\omega_\nu=\beta(\nu)^{-2}\rho^\nu$, which correspond to the Legendre coefficients $(\|w_\nu\|_V)_{\nu\in{\cal F}}$ and $(\|\widetilde w_\nu\|_V)_{\nu\in{\cal F}}$, in view of~\eqref{eq:wl2l}.

Likewise, in the case where $\omega_\nu=\xi_\nu^{1/2}$, using the factorization 
$$
\sum_{\nu\in {\cal F}} \kappa_\nu^q 
    = \prod_{j\geq 1} \sum_{n \geq 0} \left( \sum_{l=0}^r {n \choose l}\rho_j^{2l}   \right)^{-q/2},
$$
it is shown in \cite{BCDM} that the sum on the left converges if $b\in \ell^q$, provided
that $r$ was chosen large enough such that $q>\frac 2 r$.
This shows that the $\ell^p$ summability of the Hermite coefficients $(\|h_\nu\|_V)_{\nu\in{\cal F}}$
follows if the assumption~\eqref{eq:psicr} holds with 
$b=(\rho_j^{-1})_{j\geq 1}\in \ell^q(\mathbb{N})$. Note that, since the sequence $b$ can be
renormalized, we may replace~\eqref{eq:psicr} by the condition
\begin{equation}
\sup_{x\in D}\sum_{j\geq 1} \rho_j|\psi_j(x)|<\infty,
\label{eq:psicr1}
\end{equation}
without a specific bound.

\subsection{Approximation by Downward Closed Polynomials}
\label{sec:2.3}

The above results, combined with Lemma~\ref{thm:stechkin}, allow us
to build polynomial approximations $u_{\Lambda_n}$ with 
provable convergence rates $n^{-s}$ in $L^\infty$ or $L^2$
by $n$-term trunctation of the various polynomial expansions.
However, we would like to obtain such convergence rates with sets $\Lambda_n$ that are in addition downward closed.

Notice that if a sequence $(\kappa_\nu)_{\nu\in{\cal F}}$ of nonnegative numbers is
monotone nonincreasing, that is
$$
\nu\leq \widetilde \nu \implies  \kappa_{\widetilde \nu} \leq \kappa_{\nu},
$$
then the set $\Lambda_n$ corresponding to the $n$ largest values of $\kappa_\nu$ (up to a specific selection
in case of equal values) is downward closed. More generally, there exists a sequence $(\Lambda_n)_{n\geq 1}$ of 
downward closed realizations of such sets which is nested, \emph{i.e.}~$\Lambda_1 \subset \Lambda_2 \ldots $,
with $\Lambda_1=0_{{\cal F}}:=(0,0,\ldots)$.

For any sequence $(\kappa_\nu)_{\nu\in{\cal F}}$ tending to $0$, in the sense 
that $\#\{\nu \; : \; |\kappa_\nu|>\delta\} <\infty$ for all $\delta >0$,
we introduce its {\it monotone majorant} $(\widehat \kappa_{\nu})_{\nu\in{\cal F}}$
defined by
$$
\widehat \kappa_\nu:=\max_{\widetilde \nu\geq \nu} |\kappa_{\widetilde \nu}|,
$$
that is the smallest monotone nonincreasing sequence that dominates $(\kappa_\nu)_{\nu\in{\cal F}}$. In order
to study best $n$-term approximations using downward closed sets, we adapt the $\ell^q$ spaces
as follows.

\begin{definition}
\label{thm:deflpm}
For $0<q<\infty$, we say that $(\kappa_\nu)_{\nu\in{\cal F}}\in \ell^\infty({\cal F})$ 
belongs to $\ell^q_m({\cal F})$ if and only if its monotone majorant $(\widehat \kappa_\nu)_{\nu\in{\cal F}}$ 
belongs to $\ell^q({\cal F})$. 
\end{definition}

We are now in position to state a general theorem that gives a condition
for approximation using downward closed sets in terms of weighted $\ell^2$ summability.  

\begin{theorem}
\label{thm:theowlower}
Let $(c_\nu)_{\nu\in{\cal F}}$ and $(\omega_\nu)_{\nu\in{\cal F}}$ be positive sequences such that
$$
\sum_{\nu\in{\cal F}}(\omega_\nu c_\nu)^2<\infty,
$$
and such that $(\kappa_\nu)_{\nu\in{\cal F}}\in \ell^q_m({\cal F})$ for some $0<q<\infty$ with $\kappa_\nu=\omega_\nu^{-1}$. Then,
for any $0<r\leq 2$ such that $\frac 1 q>\frac 1 r-\frac 1 2$,
there exists a nested sequence $(\Lambda_n)_{n\geq 1}$ of downward closed sets such that $\#(\Lambda_n)=n$ and
\begin{equation}
\left(\sum_{\nu \notin \Lambda_n} c_\nu^r\right)^{1/r} \leq Cn^{-s}, \quad s:=\frac 1 q+\frac 1 2-\frac 1 r>0.
\label{eq:rdownwardrate}
\end{equation}
\end{theorem}

\begin{proof}
With $(\widehat \kappa_\nu)_{\nu\in{\cal F}}$ being the monotone majorant of  $(\kappa_\nu)_{\nu\in{\cal F}}$, we observe that
$$
A^2:=\sum_{\nu\in{\cal F}}(\widehat \kappa_\nu^{-1} c_\nu)^2\leq \sum_{\nu\in{\cal F}}(\kappa_\nu^{-1} c_\nu)^2=\sum_{\nu\in{\cal F}}(\omega_\nu c_\nu)^2 <\infty.
$$
We pick a nested sequence $(\Lambda_n)_{n\geq 1}$ of downward closed sets, such that $\Lambda_n$ consists of the indices corresponding to the $n$ largest $\widehat \kappa_\nu$.
Denoting by $(\widehat \kappa_n)_{n\geq 1}$ the decreasing rearrangement of $(\widehat \kappa_\nu)_{\nu\in{\cal F}}$, we observe that
$$
n \widehat \kappa_n^q \leq \sum_{j= 1}^n \widehat \kappa_j^q \leq  B^q, \quad B:=\|(\widehat \kappa_\nu)_{\nu\in{\cal F}}\|_{\ell^q}<\infty.
$$
With $p$ such that $\frac 1 p=\frac 1 r-\frac 1 2$, we find that
\begin{align*}
\left(\sum_{\nu \notin \Lambda_n} c_\nu^r\right)^{1/r}
& \leq \left( \sum_{\nu \notin \Lambda_n} (\widehat \kappa_\nu^{-1} c_\nu)^2\right)^{1/2}\left(\sum_{\nu \notin \Lambda_n} \widehat \kappa_\nu^p\right)^{1/p}\\
&\leq A \left(\widehat \kappa_{n+1}^{p-q}\sum_{\nu \notin \Lambda_n} \widehat \kappa_\nu^q\right)^{1/p}\\
&\leq A B (n+1)^{1/p-1/q},
\end{align*}
where we have used H\"older's inequality and the properties of $(\widehat \kappa_n)_{n\geq 1}$. 
This gives~\eqref{eq:rdownwardrate} with $C:=AB$. 
\end{proof}

We now would like to apply the above result with $c_\nu\in \{ \|t_\nu\|_V, \|w_\nu\|_V,\|\widetilde w_\nu\|_V,\|h_\nu\|_V\}$,
and the corresponding weight sequences $\omega_\nu\in \{\rho^\nu,\rho^\nu \beta(\nu)^{-1},\rho^\nu \beta(\nu)^{-2}, \xi_\nu^{1/2}\}$,
or equivalently $\kappa_\nu\in \{b^\nu,b^\nu \beta(\nu),b^\nu \beta(\nu)^2, \xi_\nu^{-1/2}\}$.
In the case of the Taylor series, where $\kappa_\nu=b^\nu$, we readily see that if $b_j<1$ for all $j\geq 1$, then the sequence
$(\kappa_\nu)_{\nu\in{\cal F}}$ is monotone nonincreasing, and therefore Lemma~\ref{thm:lemmaq}
shows that $b\in \ell^q$ implies $(\kappa_\nu)_{\nu\in{\cal F}}\in\ell^q_m({\cal F})$. By application of Theorem~\ref{thm:theowlower} with the value $r=1$, this leads to the following result.

\begin{theorem}
\label{thm:theotaylor}
If~\eqref{eq:uearho} holds with $(\rho_j^{-1})_{j\geq 1}\in \ell^q(\mathbb{N})$ for some $0<q<\infty$ and $\rho_j>1$ for all $j$, 
then 
$$
\|u-u_{\Lambda_n}\|_{L^\infty(U,V)}\leq Cn^{-s}, \quad s:=\frac 1 q-\frac 1 2,
$$
where $u_{\Lambda_n}$ is the truncated Taylor series and $\Lambda_n$ is any downward closed set corresponding to the $n$ largest $\kappa_\nu$. 
\end{theorem}

In the case of the Legendre series, the weight $\kappa_\nu=b^\nu \beta(\nu)$ is not monotone nonincreasing due to the presence of the algebraic factor $\beta(\nu)$. However, the following result holds.

\begin{lemma} 
\label{thm:lemmaqm}
For any $0<q<\infty$ and for any algebraic factor of the form $\alpha(\nu):=\prod_{j\geq 1}(1+c_1\nu_j)^{c_2}$ with $c_1,c_2\geq 0$,
one has 
$$
b\in \ell^q(\mathbb{N})\quad {\rm and}\quad \|b\|_{\ell^\infty}<1 \iff (\alpha(\nu) b^\nu)_{\nu\in{\cal F}} \in \ell^q_m({\cal F}).
$$
\end{lemma}

\begin{proof} The implication from right to left is a consequence of Lemma~\ref{thm:lemmaq}, and so we
concentrate on the implication from left to right. For this it suffices to find a majorant $\widetilde \kappa_\nu$ of 
$\kappa_\nu:=\alpha(\nu) b^\nu$ which is monotone nonincreasing and such that 
$(\widetilde \kappa_\nu)_{\nu\in{\cal F}} \in \ell^q({\cal F})$. We notice that for any $\tau>1$, there exists $C=C(\tau,c_1,c_2)\geq 1$ such that
$$
(1+c_1n)^{c_2}\leq C\tau^n,\quad n\geq 0.
$$
For some $J\geq 1$ and $\tau$ to be fixed further, we may thus write
$$
\kappa_\nu\leq \widetilde \kappa_\nu:= C^J \prod_{j=1}^J(\tau b_j)^{\nu_j} \prod_{j>J} (1+c_1\nu_j)^{c_2}b_j^{\nu_j}.
$$
Since $\|b\|_{\ell^\infty}<1$ we can take $\tau>1$ such that $\theta:=\tau\|b\|_{\ell^\infty}<1$.
By factorization, we find that
$$
\sum_{\nu\in {\cal F}} \widetilde \kappa_\nu^q=C^{Jq} \left(\prod_{j=1}^J \left(\sum_{n\geq 0} \theta^{qn}\right)\right) \left( \prod_{j>J} \left(\sum_{n\geq 0} (1+c_1n)^{qc_2} b_j^{nq}\right)\right).
$$
The first product is bounded by $(1-\theta^q)^{-J}$. Each factor in the second product is a converging series which is bounded by $1+cb_j^q$ for some $c>0$ that depends on $c_1$, $c_2$ and $\|b\|_{\ell^\infty}$.
It follows that this second product converges. Therefore $(\widetilde \kappa_\nu)_{\nu\in{\cal F}}$ belongs to $\ell^q({\cal F})$.

Finally, we show that $\widetilde \kappa_\nu$ is monotone nonincreasing provided that $J$ is chosen large enough. It suffices to show that $\widetilde \kappa_{\nu+e_j}\leq \widetilde \kappa_\nu$ for all $\nu\in {\cal F}$ and for all $j\geq 1$ where 
$$
e_j:=(0,\dots,0,1,0,\dots),
$$
is the Kronecker sequence of index $j$. When $j\leq J$ this is obvious since $\widetilde \kappa_{\nu+e_j}=\tau b_j\widetilde \kappa_\nu\leq \theta \widetilde \kappa_\nu \leq \widetilde \kappa_\nu$. When $j>J$, we have
$$
\widetilde \kappa_{\nu+e_j}\widetilde \kappa_{\nu}^{-1}=b_j\left(\frac {1+c_1(\nu_j+1)}{1+c_1\nu_j}  \right)^{c_2}.
$$
Noticing that the sequence  $a_n:=\left(\frac {1+c_1(n+1)}{1+c_1n}  \right)^{c_2}$ converges toward $1$ and is therefore bounded,
and that $b_j$ tends to $0$ as $j\to +\infty$, we find that for $J$ sufficiently large, the right-hand side in the above equation is 
bounded by $1$ for all $\nu$ and $j>J$. 
\end{proof}

From Lemma~\ref{thm:lemmaqm}, by applying Theorem~\ref{thm:theowlower} with $r=1$ or $r=2$, we obtain the following result.

\begin{theorem}
\label{thm:theoleg}
If~\eqref{eq:uearho} holds with 
$(\rho_j^{-1})_{j\geq 1}\in \ell^q(\mathbb{N})$ for some $0<q<\infty$ and $\rho_j>1$ for all $j$, then
$$
\|u-u_{\Lambda_n}\|_{L^2(U,V)}\leq Cn^{-s}, \quad s:=\frac 1 q,
$$
where $u_{\Lambda_n}$ is the truncated Legendre series and $\Lambda_n$ is any downward closed set corresponding to the
$n$ largest $\widehat \kappa_\nu$ where $\kappa_\nu:=b^\nu \beta(\nu)$. If $q<2$, we also have
$$
\|u-u_{\Lambda_n}\|_{L^\infty(U,V)}\leq Cn^{-s}, \quad s:=\frac 1 q-\frac 1 2,
$$
with $\Lambda_n$ any downward closed set corresponding to the $n$ largest $\widehat \kappa_\nu$ where $\kappa_\nu:=b^\nu \beta(\nu)^2$.
\end{theorem}

Finally, in the case of the Hermite coefficients, which corresponds to the weight
\begin{equation}
\kappa_\nu:= \prod_{j\geq 1}\left( \sum_{l=0}^r {\nu_j \choose l}b_j^{-2l}   \right)^{-1/2},
\label{eq:bnuherm}
\end{equation}
we can establish a similar summability result.

\begin{lemma} 
\label{thm:lemmaqmh}
For any $0<q<\infty$ and any integer $r\geq 1$ such that $q> \frac 2 r$, we have
$$
b\in \ell^q(\mathbb{N})\implies (\kappa_\nu)_{\nu\in{\cal F}} \in \ell^q
({\cal F}),
$$
where $\kappa_\nu$ is given by~\eqref{eq:bnuherm}. 
In addition, for any integer $r\geq 0$, the sequence $(\kappa_\nu)_{\nu \in {\cal F}}$ is monotone nonincreasing. 
\end{lemma}

\begin{proof} 
For any $\nu \in {\cal F}$ and any $k\geq 1$ we have 
\begin{align*}
\kappa_{\nu + e_k} 
& = 
\left(
\sum_{l=0}^r \binom{\nu_k + 1}{ l} b_k^{-2l}
\right)^{-1/2}
\prod_{j \geq 1 \atop j \neq k} \left(
\sum_{l=0}^r \binom{\nu_j}{ l} b_j^{-2l}
\right)^{-1/2}
\\
& \leq 
\left(
\sum_{l=0}^r \binom{\nu_k}{ l} b_k^{-2l}
\right)^{-1/2}
\prod_{j \geq 1 \atop j \neq k} 
\left(
\sum_{l=0}^r \binom{\nu_j}{ l} b_j^{-2l}
\right)^{-1/2}
=
\kappa_{\nu}, 
\end{align*}
and therefore the sequence $(\kappa_\nu)_{\nu \in {\cal F}}$ is monotone nonincreasing.

Now we check that $(\kappa_\nu)_{\nu\in{\cal F}} \in \ell^q({\cal F})$, using the factorization
\begin{equation}
\label{eq:factorization_proof}
\sum_{\nu\in{\cal F}} \kappa_\nu^q=\prod_{j\geq 1}  \sum_{n \geq 0}\left( \sum_{l=0}^r {n \choose l}b_j^{-2l}   \right)^{-q/2} 
\leq 
\prod_{j\geq 1}\sum_{n\geq 0}{n \choose r\wedge n}^{-q/2}b_j^{q(r\wedge n)}.
\end{equation}
where the inequality follows from the fact that the value $l=n\wedge r:=\min\{n,r\}$ is contained in the sum.

The $j$th factor $F_j$ in the rightmost product in~\eqref{eq:factorization_proof} may be written as
$$
F_j= 1+b_j^q+\cdots+b_j^{(r-1)q}+C_{r,q}b_j^{rq},
$$
where
\begin{equation}
 C_{r,q} := \sum_{n \geq r} {n\choose r}^{-q/2} = (r!)^{q/2} \sum_{n\geq 0} \bigl[(n+1)\cdots (n+r)\bigr]^{-q/2}<\infty,
\end{equation}   
since we have assumed that $q> 2/r$. This shows that each $F_j$ is finite.
If $b\in \ell^q(\N)$, there exists an integer $J\geq 0$ such that $b_j <1$ for all $j> J$. For such $j$, we can bound $F_j$
by $1 + (C_{r,q} + r - 1) b_j^{q}$, which shows that the product converges.
\end{proof}

From this lemma, and by application of Theorem~\ref{thm:theowlower} with the value $r=2$, we obtain the following result for the Hermite series.

\begin{theorem}
\label{thm:theoherm}
If~\eqref{eq:psicr1} holds with $(\rho_j^{-1})_{j\geq 1}\in \ell^q(\mathbb{N})$ for some $0<q<\infty$, then
$$
\|u-u_{\Lambda_n}\|_{L^2(U,V)}\leq Cn^{-s}, \quad s:=\frac 1 q,
$$
where $u_{\Lambda_n}$ is the truncated Hermite series and $\Lambda_n$ is a downward closed set corresponding to the $n$ largest $\kappa_\nu$ given by~\eqref{eq:bnuherm}.
\end{theorem}

In summary, we have established convergence rates for approximation
by downward closed polynomial spaces of the solution map~\eqref{eq:solmap}
associated to the elliptic PDE~\eqref{eq:ellip} with affine or lognormal parametrization.
The conditions are stated in terms of the control on the $L^\infty$ norm
of $\sum_{j\geq 1} \rho_j|\psi_j|$, where the $\rho_j$ have a certain growth
measured by the $\ell^q$ summability of the sequence $b=(b_j)_{j\geq 1}=(\rho_j^{-1})_{j\geq 1}$. This 
is a way to quantify the decay of the size of the $\psi_j$, also taking
their support properties into account, and in turn to quantify the anisotropic dependence
of $u(y)$ on the various coordinates $y_j$.
Other similar results have been obtained with different PDE models, see 
in particular \cite{CD}. In the above results, the polynomial approximants are constructed
by truncation of infinite series. The remainder of the paper addresses the construction
of downward closed polynomial approximants from evaluations of the solution
map at $m$ points $\{y^1,\dots,y^m\}\in U$, and discusses the accuracy of these approximants.

\section{Interpolation}
\label{sec:3}

\subsection{Sparse Interpolation by Downward Closed Polynomials}
\label{sec:3.1}

Interpolation is one of the most standard processes for constructing polynomial approximations based on pointwise evaluations. Given a downward closed set $\Lambda \subset {\cal F}$ of finite cardinality, and a set of points 
$$
\Gamma\subset  U, \quad \#(\Gamma)=\#(\Lambda),
$$
we would like to build an interpolation operator $I_\Lambda$, that is, $I_\Lambda u\in \mathbb{V}_\Lambda$ is uniquely characterized by
$$
I_\Lambda u(y)=u(y), \quad y\in\Gamma,
$$
for any $V$-valued function $u$ defined on $U$. 

In the univariate case, it is well known that such an operator exists if and only if $\Gamma$ is a set of pairwise distinct points, and that additional conditions are needed in the multivariate case. Moreover, since the set $\Lambda$ may come from a nested sequence $(\Lambda_n)_{n\geq 1}$ 
as discussed in Section~\ref{sec:2}, we are interested in having similar nestedness properties for the corresponding sequence $(\Gamma_n)_{n\geq 1}$, where 
$$
\#(\Gamma_n)=\#(\Lambda_n)=n.
$$
Such a nestedness property allows us to recycle the $n$ evaluations of $u$ which have been used in the computation of $I_{\Lambda_n} u$,
and use only one additional evaluation for the next computation of $I_{\Lambda_{n+1}} u$.

It turns out that such hierarchical interpolants can be constructed in a natural manner by making use of the downward closed
structure of the index sets $\Lambda_n$. This construction is detailed in  \cite{CCS1}
but its main principles can be traced from \cite{K}. In order to describe
it, we assume that the parameter domain is of either form
$$
U=[-1,1]^d \quad {\rm or} \quad [-1,1]^\mathbb{N},
$$
with the convention that $d=\infty$ in the second case. However, it is easily check
that the construction can be generalized in a straightforward manner 
to any domain with Cartesian product form
$$
U=\underset{k\geq 1}{\times} J_k,
$$
where the $J_k$ are finite or infinite intervals.

We start from a sequence of pairwise distinct points
$$
T=(t_k)_{k\geq 0}\subset [-1,1].
$$
We denote by $I_k$ the univariate interpolation operator on the space
$
\mathbb{V}_k:=V\otimes \mathbb{P}_k
$
associated with the $k$-section $\{t_0,\dots,t_k\}$ of this 
sequence, that is, 
$$
I_ku(t_i)=u(t_i),\quad i=0,\dots,k,
$$
for any $V$-valued function $u$ defined on $[-1,1]$. 
We express $I_k$ in the Newton form
\begin{equation}
I_ku=I_0u+\sum_{l=1}^k\Delta_l u,\quad \Delta_l:=I_l-I_{l-1},
\label{eq:newton}
\end{equation}
and set $I_{-1}=0$ so that we can also write
$$
I_ku=\sum_{l=0}^k \Delta_l u.
$$
Obviously the difference operator $\Delta_k$ annihilates the elements of $\mathbb{V}_{k-1}$.
In addition, since $\Delta_k u(t_j)=0$ for $j=0,\dots,k-1$, we have
$$
\Delta_k u(t)=\alpha_k B_k(t),
$$
where 
$$
B_{k}(t):=\prod_{l=0}^{k-1}\frac {t-t_l}{t_k-t_l}.
$$
The coefficient $\alpha_k\in V$ can be computed inductively, since it is given by
$$
\alpha_k=\alpha_k(u):=u(t_k)-I_{k-1}u(t_k),
$$
that is, the interpolation error at $t_k$ when using $I_{k-1}$.
Setting
$$
B_0(t):=1,
$$
we observe that the system $\{B_0,\dots,B_k\}$ is a basis for $\mathbb{P}_k$.
It is sometimes called a {\it hierarchical basis}.

In the multivariate setting, we tensorize the grid $T$, by defining
$$
y_\nu:=(t_{\nu_j})_{j\geq 1}\in U, \quad \nu\in{\cal F}.
$$
We first introduce the tensorized operator
$$
I_\nu:=\bigotimes_{j\geq 1} I_{\nu_j},
$$
recalling that the application of a tensorized operator $\otimes_{j\geq 1} A_j$ to a multivariate function amounts in applying each univariate operator $A_j$ by freezing all variables except the $j$th one, and then applying $A_j$ to the unfrozen variable.
It is readily seen that $I_\nu$ is the interpolation operator on the tensor product polynomial space 
$$
\mathbb{V}_\nu=V\otimes \mathbb{P}_\nu, \quad  \mathbb{P}_\nu:=\bigotimes_{j\geq 1} \mathbb{P}_{\nu_j},
$$
associated to the grid of points
$$
\Gamma_{\nu}=\underset{j\geq 1}{\times}
\{t_{0},\dots,t_{\nu_j}\}.
$$
This polynomial space corresponds to the particular downward closed index set of rectangular shape
$$
\Lambda=R_\nu:=\{ \widetilde \nu \; : \; \widetilde \nu \leq \nu \}.
$$
Defining in a similar manner the tensorized 
difference operators
$$
\Delta_\nu:=\bigotimes_{j\geq 1}\Delta_{\nu_j},
$$
we observe that 
$$
I_\nu=\bigotimes_{j\geq 1} I_{\nu_j}=\bigotimes_{j\geq 1} (\sum_{l=0}^{\nu_j} \Delta_l)=\sum_{\widetilde\nu\in R_\nu} \Delta_{\widetilde\nu}.
$$
The following result from \cite{CCS1} shows that the above formula can be generalized to {\it any} downward closed set
in order to define an interpolation operator. We recall its proof for sake of completeness.

\begin{theorem}
\label{thm:Thmlowerint}
Let $\Lambda\subset {\cal F}$ be a finite downward closed set, and define the grid
$$
\Gamma_\Lambda:=\{y_\nu \; : \; \nu\in \Lambda\}.
$$
Then, the interpolation operator onto $\mathbb{V}_\Lambda$ for this grid is defined by
\begin{equation}
I_\Lambda:=\sum_{\nu\in \Lambda} \Delta_\nu.
\label{eq:newtongen}
\end{equation}
\end{theorem}

\begin{proof} 
From the downward closed set property,  $\mathbb{V}_\nu\subset \mathbb{V}_\Lambda$ for all $\nu\in\Lambda$. 
Hence the image of $I_\Lambda$  is contained in $\mathbb{V}_\Lambda$. 
With $I_\Lambda$ defined by~\eqref{eq:newtongen}, we may write
$$
I_\Lambda u=I_{\nu}u+\sum_{\widetilde \nu\in\Lambda, \widetilde \nu \nleqslant \nu} \Delta_{\widetilde \nu} u,
$$
for any $\nu\in \Lambda$. Since $y_\nu\in \Gamma_\nu$, we know that
$$
I_\nu u(y_\nu)=u(y_\nu).
$$
On the other hand, if $\widetilde \nu \nleqslant \nu$, this means that
there exists a $j\geq 1$ such that $\widetilde \nu_j>\nu_j$. For this $j$ we thus have
$\Delta_{\widetilde \nu} u(y)=0$ for all $y\in U$ with the $j$th coordinate  equal to $t_{\nu_j}$ 
by application of $\Delta_{\nu_j}$ in the $j$th variable, so that
$$
\Delta_{\widetilde \nu} u(y_\nu)=0.
$$
The interpolation property $I_\Lambda u(y_\nu)=u(y_\nu)$ thus holds, for all $\nu\in \Lambda$.
\end{proof}

The decomposition~\eqref{eq:newtongen} should be viewed
as a generalization of the Newton form~\eqref{eq:newton}. In a similar way,
its terms can be computed inductively: if $\Lambda=\widetilde \Lambda\cup \{\nu\}$ 
where $\widetilde \Lambda$ is a downward closed set, we have
$$
\Delta_\nu u=\alpha_\nu B_\nu, 
$$
where 
$$
B_\nu(y):=\prod_{j\geq 1} B_{\nu_j}(y_j),
$$
and
$$
\alpha_\nu=\alpha_\nu(u):=u(y_\nu)-I_{\widetilde \Lambda} u(y_\nu).
$$
Therefore, if $(\Lambda_n)_{n\geq 1}$ is any nested sequence of downward closed
index sets, we can compute $I_{\Lambda_n}$ by $n$ iterations of
$$
I_{\Lambda_{i}}u=I_{\Lambda_{i-1}}u +\alpha_{\nu^i} B_{\nu^i},
$$
where $\nu^i\in \Lambda_i$ is such that $\Lambda_i=\Lambda_{i-1} \cup\{\nu^i\}$.

Note that $(B_\nu)_{\nu\in \Lambda}$ is a basis of $\mathbb{P}_\Lambda$ and that
any $f \in \mathbb{V}_\Lambda$ has the unique decomposition
$$
f=\sum_{\nu\in \Lambda} \alpha_\nu B_\nu,
$$
where the coefficients $\alpha_\nu=\alpha_\nu(f)\in V$ are defined by the above procedure. 
Also note that $\alpha_\nu(f)$ does not depend on the choice of $\Lambda$
but only on $\nu$ and $f$.

\subsection{Stability and Error Estimates}
\label{sec:3.2}

The pointwise evaluations of the function $u$ 
could be affected by errors, as modeled by~\eqref{eq:noisy}
and~\eqref{eq:noisebound}.
The stability of the interpolation operator with respect to such 
perturbations is quantified by the
{\em Lebesgue constant}, which is defined by
$$
\mathbb{L}_\Lambda:=\sup \frac{\|I_\Lambda f\|_{L^\infty(U,V)}} {\|f\|_{L^\infty(U,V)}},
$$
where the supremum is taken over the set of all $V$-valued functions $f$ defined everywhere and uniformly bounded over $U$. It is easily seen
that this supremum is in fact independent of the space $V$, so that we may also write
$$
\mathbb{L}_\Lambda:=\sup \frac{\|I_\Lambda f\|_{L^\infty(U)}} {\|f\|_{L^\infty(U)}},
$$
where the supremum is now taken over real-valued functions.  Obviously, we have
$$
\|u-I_\Lambda(u+\eta)\|_{L^\infty(U,V)} \leq \|u-I_\Lambda u\|_{L^\infty(U,V)}+ \mathbb{L}_\Lambda \varepsilon,
$$
where $\varepsilon$ is the noise level from~\eqref{eq:noisebound}. 

The Lebesgue constant
also allows us to estimate the error
of interpolation $\|u-I_\Lambda u\|_{L^\infty(U,V)}$ 
for the noiseless solution map
in terms of the best approximation error in the $L^\infty$ norm:
for any $u\in L^\infty(U,V)$ and any $\widetilde u\in \mathbb{V}_\Lambda$ we have
$$
\|u-I_\Lambda u\|_{L^\infty(U,V)} \leq \|u-\widetilde u\|_{L^\infty(U,V)}+\|I_\Lambda \widetilde u-I_\Lambda u\|_{L^\infty(U,V)},
$$
which by infimizing over $\widetilde u\in \mathbb{V}_\Lambda$ yields
$$
\|u-I_\Lambda u\|_{L^\infty(U,V)} \leq (1+\mathbb{L}_\Lambda)\inf_{\widetilde u\in \mathbb{V}_\Lambda}\|u-\widetilde u\|_{L^\infty(U,V)}.
$$
We have seen in Section~\ref{sec:2} that for relevant classes of solution maps $y\mapsto u(y)$, there exist sequences of downward closed sets $(\Lambda_n)_{n\geq 1}$ with $\#(\Lambda_n)=n$, such that
$$
\inf_{\widetilde u\in \mathbb{V}_{\Lambda_n}}\|u- \widetilde u\|_{L^\infty(U,V)} \leq Cn^{-s}, \quad n\geq 1,
$$
for some $s>0$. For such sets, we thus have 
\begin{equation}
\|u-I_{\Lambda_n} u\|_{L^\infty(U,V)} \leq C(1+\mathbb{L}_{\Lambda_n})n^{-s}.
\label{eq:Intern}
\end{equation}
This motivates the study of the growth of $\mathbb{L}_{\Lambda_n}$ as $n\to +\infty$.

For this purpose, we introduce the univariate Lebesgue constants
$$
\mathbb{L}_k:=\sup\frac {\|I_k f\|_{L^\infty([-1,1])}} {\|f\|_{L^\infty([-1,1])}}.
$$
Note that $\mathbb{L}_0=1$. We also define an analog quantity for the difference operator
$$
\mathbb{D}_k:=\sup\frac {\|\Delta_k f\|_{L^\infty([-1,1])}} {\|f\|_{L^\infty([-1,1])}}.
$$
In the particular case of the rectangular downward closed sets
$\Lambda=R_\nu$, since $I_\Lambda=I_\nu=\otimes_{j\geq 1}I_{\nu_j}$, we have 
$$
\mathbb{L}_{R_\nu}=\prod_{j\geq 1} \mathbb{L}_{\nu_j}.
$$
Therefore, if the sequence $T=(t_k)_{k\geq 0}$ is such that
\begin{equation}
\mathbb{L}_k \leq (1+k)^\theta ,\quad k\geq 0,
\label{eq:lambdak}
\end{equation}
for some $\theta \geq 1$, we find that
$$
\mathbb{L}_{R_\nu}\leq \prod_{j\geq 1} (1+\nu_j)^\theta= (\#(R_\nu))^\theta,
$$
for all $\nu\in{\cal F}$.

For arbitrary downward closed sets $\Lambda$, the expression of $I_\Lambda$ shows that
$$
\mathbb{L}_\Lambda\leq \sum_{\nu\in \Lambda} \prod_{j\geq 1} \mathbb{D}_{\nu_j}.
$$
Therefore, if the sequence  $T=(t_k)_{k\geq 0}$ is such that
\begin{equation} 
\mathbb{D}_k \leq (1+k)^\theta ,\quad k\geq 0,
\label{eq:dak}
\end{equation}
we find that
$$
\mathbb{L}_\Lambda \leq \sum_{\nu\in \Lambda} \prod_{j\geq 1} (1+\nu_j)^\theta
=\sum_{\nu\in \Lambda} (\#(R_\nu))^\theta \leq \sum_{\nu\in \Lambda} (\#(\Lambda))^\theta = (\#(\Lambda))^{\theta+1}.
$$
The following result from \cite{CCS1} shows that this general estimate is also valid under the assumption~\eqref{eq:lambdak} on the growth of $\mathbb{L}_k$.

\begin{theorem} 
\label{thm:theoLeb}
If the sequence $T=(t_k)_{k\geq 0}$ is such that~\eqref{eq:lambdak} or~\eqref{eq:dak} holds for some $\theta \geq 1$, then
$$
\mathbb{L}_{\Lambda} \leq (\#(\Lambda))^{\theta+1},
$$
for all downward closed sets $\Lambda$.
\end{theorem}

One noticeable feature of the above result is that the bound on $\mathbb{L}_\Lambda$ only depends on $\#(\Lambda)$, independently of
the number of variables, which can be infinite, as well as of the shape of $\Lambda$. 

We are therefore interested in univariate sequences
$T=(t_k)_{k\geq 0}$ such that $\mathbb{L}_k$ and $\mathbb{D}_k$ have moderate growth
with $k$. For Chebyshev or Gauss-Lobatto points, given by
$$
{\cal C}_k:=\left\{\cos\left(\frac {2 l+1}{2k+2} \pi\right)\; : \; l=0,\dots, k\right\} \quad {\rm and}\quad
{\cal G}_k:=\left\{\cos\left(\frac {l}{k} \pi\right)\; : \; l=0,\dots, k\right\},
$$
it is well known that the Lebesgue constant has logarithmic growth $\mathbb{L}_k\sim \ln(k)$, thus slower than algebraic.
However these points are not the $k$ section of a single sequence $T$, and therefore they are not convenient
for our purposes. Two examples of univariate sequences of interest are the following.
\begin{itemize}
\item
The {\it Leja points}: from an arbitrary $t_0\in [-1,1]$ (usually taken to be $1$ or $0$), this sequence
is recursively defined by
$$
t_k:=\argmax \left\{ \prod_{l=0}^{k-1}|t-t_l| \; : \; t\in [-1,1]\right\}.
$$
Note that this choice yields hierarchical basis functions $B_k$
that are uniformly bounded by $1$. Numerical computations of $\mathbb{L}_k$ for the first $200$ values of $k$ 
indicates that the linear
bound 
\begin{equation}
\mathbb{L}_k \leq 1+k,
\label{eq:lineark}
\end{equation}
holds. Proving that this bound, or any other algebraic growth bound, holds for all values of $k\geq 0$ is currently an open problem.
\item
The {\it $\Re$-Leja}  points: they are the real part of the Leja points defined on the complex unit disc $\{|z|\leq 1\}$, taking for example $e_0=1$ and recursively setting
$$
e_k:=\argmax \left\{ \prod_{l=0}^{k-1}|e-e_l| \; : \; |e|\leq 1\right \}.
$$
These points have the property of accumulating in a regular manner on the 
unit circle according to the so-called Van der Corput enumeration \cite{CP1}. 
It is proven in \cite{C} that the linear bound~\eqref{eq:lineark}
holds for the Lebesgue constant of the complex interpolation operator on the unit disc 
associated to these points. The sequence of real parts
$$
t_k:=\Re(e_k),
$$
is defined after eliminating the possible repetitions corresponding to $e_k=\overline e_l$ for
two different values of $k=l$. These points coincide 
with the Gauss-Lobatto points for values of $k$ of the form $2^n+1$ for $n\geq 0$.
A quadratic bound 
$$
\mathbb{D}_k \leq (1+k)^2,
$$
is established in \cite{CC}. 
\end{itemize}

If we use such sequences, application of Theorem~\ref{thm:theoLeb}
gives bounds of the form
$$
\mathbb{L}_\Lambda\leq (\#(\Lambda))^{1+\theta},
$$
for example with $\theta=2$ when using the $\Re$-Leja points, or $\theta=1$
when using the Leja points provided that the conjectured bound~\eqref{eq:lineark} holds. 
Combining with~\eqref{eq:Intern}, we obtain the convergence
estimate
$$
\|u-I_{\Lambda_n} u\|_{L^\infty(U,V)} \leq C n^{-(s-1-\theta)},
$$
which reveals a serious deterioration of the convergence rate when using interpolation instead of truncated expansions.

However, for the parametric PDE models discussed in Section~\ref{sec:2}, it is possible to show that this deterioration actually does not occur, based on the following lemma which relates the interpolation error to the summability of coefficient sequences in general expansions of $u$.

\begin{lemma} 
\label{thm:lemmaresid}
Assume that $u$ admits an
expansion of the type~\eqref{eq:expan}, where $\|\phi_\nu\|_{L^\infty(U)}\leq 1$
which is unconditionally convergent towards $u$ in $L^\infty(U,V)$. Assume in addition
that $y\mapsto u(y)$ is continuous from $U$ equipped with the 
product topology toward $V$. If the univariate sequence $T=(t_k)_{k\geq 0}$ is such 
that~\eqref{eq:lambdak} or~\eqref{eq:dak} holds for some $\theta\geq 1$, 
then, for any downward closed set $\Lambda$,
\begin{equation}
\| u-I_\Lambda u \|_{L^\infty(U,V)} 
\leq 
2\sum_{\nu\notin \Lambda} \pi (\nu) \|u_\nu\|_V, \quad \pi(\nu) := \prod_{j\geq 1} (1+\nu_j)^{\theta+1}.
\label{eq:residual}
\end{equation}
\end{lemma}
\begin{proof}
The unconditional convergence of~\eqref{eq:expan} and the continuity of $u$
with respect to the product topology allow us to say that 
the equality in~\eqref{eq:expan} holds everywhere in $U$. We
may thus write
$$
I_\Lambda u 
=
I_\Lambda\left(\sum_{\nu\in{\cal F}} u_\nu \phi_\nu\right)
=
 \sum_{\nu\in{\cal F}} u_\nu I_\Lambda \phi_\nu
=
\sum_{\nu\in\Lambda} u_\nu \phi_\nu
+
\sum_{\nu\notin\Lambda} u_\nu I_{\Lambda}\phi_\nu,
$$
where we have used that
$I_\Lambda \phi_\nu =\phi_\nu$ for every $\nu\in\Lambda$ since $\phi_\nu\in \mathbb{P}_\Lambda$.
For the second sum on the right-hand side, we observe that for each $\nu\notin\Lambda$,
$$
I_{\Lambda}\phi_\nu=\sum_{\widetilde \nu\in \Lambda}\Delta_{\widetilde \nu} \phi_\nu
=\sum_{\widetilde \nu\in \Lambda\cap R_\nu}\Delta_{\widetilde \nu} \phi_\nu=I_{\Lambda\cap R_\nu} \phi_\nu,
$$
since $\Delta_{\widetilde \nu}$ annihilates $\mathbb{P}_\nu$ whenever $\widetilde \nu \not\leq \nu$.
Therefore 
$$
u-I_\Lambda u  
= \sum_{\nu\not\in\Lambda} u_\nu (I- I_{\Lambda\cap R_\nu })\phi_\nu,
$$
where $I$ stands for the identity operator. 
This implies
$$
\|u-I_\Lambda u \|_{L^\infty(U,V)}
\leq
\sum_{\nu\not\in\Lambda} (1+\mathbb{L}_{\Lambda\cap R_\nu})\|u_\nu\|_V 
\leq 
2\sum_{\nu\not\in\Lambda} \mathbb{L}_{\Lambda\cap R_\nu}\|u_\nu\|_V \; .
$$
Since~\eqref{eq:lambdak} or~\eqref{eq:dak} holds, we obtain from Theorem~\ref{thm:theoLeb} that
$$
\mathbb{L}_{\Lambda\cap R_\nu} 
\leq (\#(\Lambda\cap R_\nu))^{\theta+1} 
\leq (\#(R_\nu))^{\theta+1} 
= \pi(\nu),
$$
which yields~\eqref{eq:residual}.
\end{proof}

We can apply the above lemma with the Taylor series~\eqref{eq:taylor} or the renormalized
Legendre series~\eqref{eq:renormlegendre}. This leads us to analyze the $\ell^1$ tail of the
sequence $(c_\nu)_{\nu\in{\cal F}}$ where $c_\nu$ is either $\pi(\nu)\|t_\nu\|_V$ or $\pi(\nu)\|\widetilde w_\nu\|_V$.
If~\eqref{eq:uearho} holds, we know from Theorem~\ref{thm:theosumaff} that this sequence satisfies the bound
$$
\sum_{\nu\in{\cal F}}(\omega_\nu c_\nu)^2 <\infty,
$$
where $\omega_\nu$ is either $\pi(\nu)^{-1}\rho^\nu$ or $\pi(\nu)^{-1}\beta(\nu)^{-2}\rho^\nu$.
Since $\pi(\nu)$ has algebraic growth similar to $\beta(\nu)$, application
of Lemma~\ref{thm:lemmaqm} and of Theorem~\ref{thm:theowlower} with the value $r=1$, leads to the following 
result.

\begin{theorem}
\label{thm:theointer}
If~\eqref{eq:uearho} holds with $(\rho_j^{-1})_{j\geq 1}\in \ell^q(\mathbb{N})$ for some $0<q<\infty$ and $\rho_j>1$ for all $j$, then 
$$
\|u-I_{\Lambda_n}u\|_{L^\infty(U,V)}\leq Cn^{-s}, \quad s:=\frac 1 q-\frac 1 2,
$$
where $\Lambda_n$ is any downward closed set corresponding to the $n$ largest $\widehat \kappa_\nu$ where $\kappa_\nu$ is either $\pi(\nu)b^\nu$ or $\pi(\nu) \beta(\nu)^2 b^\nu$.
\end{theorem}

\section{Discrete Least Squares Approximations}
\label{sec:4}

\subsection{Discrete Least Squares on $V$-valued Linear Spaces}
\label{sec:4.1}

Least-squares fitting is an alternative approach to interpolation for building a polynomial approximation of $u$ from $\mathbb{V}_\Lambda$.
In this approach we are given $m$ observations $u^1,\ldots,u^m$ of $u$ at points $y^1,\ldots,y^m \in U \subseteq \mathbb{R}^d$ where $m\geq n=\#(\Lambda)$.

We first discuss the least-squares method in the more general setting of $V$-valued linear spaces, 
$$
\mathbb{V}_n:=V\otimes \mathbb{Y}_n,
$$
where $\mathbb{Y}_n$ is the space of real-valued functions defined everywhere 
on $U$ such that $\dim(\mathbb{Y}_n)=n$.  In the next section, we discuss more
specifically the case where $\mathbb{Y}_n=\mathbb{P}_\Lambda$. Here we study the approximation error in the $L^2(U,V,d\mu)$ norm for some given probability measure $d\mu$, when the evaluation points $y^i$ are independent and drawn according to this probability measure. For notational simplicity we use the shorthand
$$
\|\cdot\|:=\|\cdot \|_{L^2(U,V,d\mu)}.
$$

The {\it least-squares} method selects 
the approximant of $u$ in the space $\mathbb{V}_n$ as
$$
u_L:=\argmin_{\widetilde u \in \mathbb{V}_n} \frac 1 m \sum_{i=1}^m  \| \widetilde u(y^i)-u^i \|_V^2.
$$
In the noiseless case where $u^i:=u(y^i)$ for any $i=1,\ldots,m$, this also writes
\begin{equation}
u_L=\argmin_{\widetilde u \in  \mathbb{V}_\Lambda }\|u-\widetilde u\|_m,
\label{eq:dls}
\end{equation}
where the discrete seminorm is defined by 
$$
\| f \|_m:= \left( \frac 1 m \sum_{i=1}^m   \|f(y^i) \|_V^2 \right)^{1/2}.
$$
Note that $\|f\|_m^2$ is an unbiased estimator of $\|f\|^2$ since we have
$$
\mathbb{E}(\|f\|_m^2)=\|f\|^2.
$$
Let $\{\phi_1,\dots,\phi_n\}$ denote an arbitrary $L^2(U,d\mu)$ orthonormal basis  of the space $\mathbb{Y}_n$. 
If we expand the solution to~\eqref{eq:dls} as $\sum_{j=1}^n c_j \phi_j$, with $c_j\in V$, the 
$V$-valued vector ${\bf c}=(c_1,\dots,c_n)^t$ is the solution to the normal equations
\begin{equation}
{\bf G} {\bf c}={\bf d},
\label{eq:sys}
\end{equation}
where the matrix ${\bf G}$ has entries
$$
{\bf G}_{j,k}=\frac 1 m\sum_{i=1}^m \phi_j(y^i)\phi_k(y^i),
$$
and where
the $V$-valued data vector ${\bf d}=(d_1,\dots,d_n)^t$ is given by 
$$
d_j:=\frac 1 m \sum_{i=1}^m  u^i \phi_j(y^i).
$$
This linear system always has at least one solution, which is unique when ${\bf G}$ is nonsingular.
When ${\bf G}$ is singular, we may define $u_L$ as the unique minimal $\ell^2(\mathbb{R}^n,V)$ norm solution to~\eqref{eq:sys}.

In the subsequent analysis, we sometimes work under the assumption of a known uniform bound 
\begin{equation}
\| u \|_{L^\infty(U,V)}\leq \tau.
\label{eq:unibound}
\end{equation}
We introduce the truncation operator
$$
z \mapsto T_{\tau} (z):=
\left\{
\begin{array}{ll}
z,  
&  
\ 
\textrm{ if }  \ \| z\|_V \leq \tau, 
\\
\frac{z}{\|z \|_V},
&
\ 
\textrm{ if }  \ \| z\|_V > \tau,
\end{array}
\right.
$$
and notice that it is a contraction: $\| T_{\tau}(z)-T_{\tau}(\widetilde z)\|_V \leq \|z -\widetilde z\|_V $ for any $z,\widetilde z \in V$. 
The {\it truncated  least-squares approximation} is defined by 
$$
u_T:=T_{\tau}\circ u_L.
$$
Note that, in view of~\eqref{eq:unibound}, we have $\| u(y) - u_T(y) \|_V\leq \| u(y) -u_L(y) \|_V$ for any $y\in U$ and therefore
$$
\|u - u_T\|\leq \|u -u_L\|.
$$
Note that the random matrix ${\bf G}$ concentrates toward its expectation which is the identity matrix ${\bf I}$ as $m\to \infty$. In other words, the probability that ${\bf G}$ is ill-conditioned becomes very small as $m$ increases.  
The truncation operator aims at avoiding instabilities which may occur when ${\bf G}$ is ill-conditioned. As an alternative proposed in \cite{CM2016}, we may define for some given $A>1$ the {\it conditioned least-squares approximation} by
$$
u_C:=u_L, \ {\rm if} \ {\rm cond}({\bf G}) \leq A, \quad u_C:=0, \ {\rm otherwise},
$$
where ${\rm cond}({\bf G}):={\lambda_{\max}({\bf G})}/{\lambda_{\min}({\bf G})}$ is the usual condition number.

The property that $\|{\bf G}-{\bf I}\|_2 \leq \delta$ for some $0<\delta<1$ amounts to the norm equivalence
$$
(1-\delta)\|f \|^2\leq \|f\|_m^2\leq (1+\delta)\|f\|^2, \quad  f\in \mathbb{V}_n.
$$
It is well known that if $m\geq n$ is too much close to $n$, least-squares methods may become unstable and inaccurate for most sampling distributions. For example, if $U=[-1,1]$ and $\mathbb{Y}_n= \mathbb{P}_{n-1}$ is the space of algebraic polynomials of degree $n-1$, then with $m=n$ the estimator coincides with the Lagrange polynomial interpolation which can be highly unstable and inaccurate, in particular for equispaced points. 
Therefore, $m$ should be sufficiently large compared to $n$ for the probability that ${\bf G}$ is ill-conditioned to be small. This trade-off between $m$ and $n$ has been analyzed in \cite{CDL}, using the function
$$
y \mapsto k_n(y):=\sum_{j=1}^n |\phi_j(y)|^2,
$$
which is the diagonal of the integral kernel of the $L^2(U,d\mu)$ projector on $\mathbb{Y}_n$.
This function depends on $d\mu$, but not on the chosen orthonormal basis. It is strictly positive in $U$ under minimal assumptions on the orthonormal basis, for example if one element of the basis is the constant function over all $U$. 
Obviously, the function $k_n$ satisfies
$$
\int_{U} k_n \, d\mu=n.
$$
We define
$$
K_n:=\|k_n \|_{L^\infty(U)} \geq n.
$$
The following results for the least-squares method with noiseless evaluations were obtained in \cite{CDL,MNST,CCMNT,CM2016} for real-valued functions, however their proof extends in a straightforward manner to the present setting of $V$-valued functions. 
They are based on a probabilistic bound for the event $\|{\bf G}-{\bf I}\|_2 > \delta$ using the particular value $\delta=\frac 1 2$, or equivalently the value $A=\frac{1+\delta}{1-\delta}=3$ as a bound on the condition number of ${\bf G}$. 

\begin{theorem}
\label{thm:theo1}
For any $r>0$, if $m$ and $n$ satisfy 
\begin{equation}
K_n \leq \kappa\frac m {\ln m},\;\; {\rm with}\;\; \kappa:=\kappa(r)=\frac{1-\ln 2} {2+2r},
\label{eq:condm}
\end{equation}
then the following hold.
\begin{enumerate}
\item[{\rm (i)}] The matrix ${\bf G}$ satisfies the tail bound
$$
{\rm Pr} \, \left\{ \|{\bf G}-{\bf I}\|_2 > \frac 1 2 \right\} \leq
2m^{-r}. 
$$
\item[{\rm (ii)}] If $u$ satisfies~\eqref{eq:unibound}, then
the truncated least-squares estimator satisfies, in the noiseless case,
$$
\mathbb{E}(\|u-u_T\|^2)\leq (1+\zeta(m))\inf_{\widetilde u\in \mathbb{V}_n}\|u-\widetilde u\|^2+8\tau^2m^{-r},
$$
where $\zeta(m):= \frac {4\kappa} {\ln (m)}\to 0$ as $m\to +\infty$, and $\kappa$ is as in~\eqref{eq:condm}.
\item[{\rm (iii)}] The conditioned least-squares estimator satisfies, in the noiseless case,
$$
\mathbb{E}(\|u-u_C\|^2)\leq (1+\zeta(m))\inf_{\widetilde u\in \mathbb{V}_n}\|u-\widetilde u\|^2+2 \| u \|^2 m^{-r},
$$
where $\zeta(m)$ is as in (ii).
\item[{\rm (iv)}] If $u$ satisfies~\eqref{eq:unibound}, then the estimator $u_E\in \{u_L,u_T,u_C\}$ satisfies, in the noiseless case,
\begin{equation}
\|u- u_E\| \leq (1+\sqrt 2) \inf_{\widetilde u\in \mathbb{V}_n}\|u-\widetilde u\|_{L^\infty(U,V)},
\label{eq:nearoptimalprob}
\end{equation}
with probability larger than $1-2m^{-r}$.
\end{enumerate}
\end{theorem}

In the case of noisy evaluations modeled by~\eqref{eq:noisy}--\eqref{eq:noisebound}, the observations are given by
\begin{equation}
u^i=u(y^i)+\eta(y^i). 
\label{eq:noisy_observations}
\end{equation}
The following result from \cite{CCMNT} shows that~\eqref{eq:nearoptimalprob} holds up to this additional perturbation.

\begin{theorem}
\label{thm:theo4}
For any $r>0$, if $m$ and $n$ satisfy condition~\eqref{eq:condm} and $u$ satisfies~\eqref{eq:unibound}, 
then the estimator $u_E\in \{u_L,u_T,u_C\}$ in the noisy case~\eqref{eq:noisy_observations} satisfies
$$
\|u- u_E\|\leq (1+\sqrt 2)
\inf_{\widetilde u \in \mathbb{V}_n}\|u-\widetilde u\|_{L^\infty(U,V)}
+ \sqrt 2 \varepsilon,
$$
with probability larger than $1-2n^{-r}$, where $\varepsilon$ is the noise level in~\eqref{eq:noisebound}.
\end{theorem}

Similar results, with more general assumptions on the type of noise, are proven in \cite{CDL,MNT2015,CM2016}. 

\subsection{Downward Closed Polynomial Spaces and Weighted Least Squares}
\label{sec:4.2}
Condition~\eqref{eq:condm} shows that $K_n$ gives indications on the number $m$ of observations required to ensure stability and accuracy of the least-squares approximation. In order to understand how demanding this condition is with respect to $m$, it is important to have sharp upper bounds for $K_n$. Such bounds have been proven when the measure $d\mu$ on $U=[-1,1]^d$ has the form 
\begin{equation}
\label{eq:jacobi_measure}
d\mu= C \bigotimes_{j=1}^d  (1-y_j)^{\theta_1}(1+y_j)^{\theta_2}  dy_j,  
\end{equation}
where $\theta_1,\theta_2>-1$ are real shape parameters and $C$ is a normalization constant such that $\int_U d\mu=1$. 
Sometimes~\eqref{eq:jacobi_measure} is called the Jacobi measure, because the Jacobi polynomials are orthonormal in $L^2(U,d\mu)$. 
Remarkable instances of the measure~\eqref{eq:jacobi_measure} are the uniform measure, when $\theta_1=\theta_2=0$, and the Chebyshev measure, when $\theta_1=\theta_2=-\frac12$.  

When $\mathbb{Y}_n=\mathbb{P}_\Lambda$ is a multivariate polynomial space and $\Lambda$ is a downward closed multi-index set with $\#(\Lambda)=n$, it is proven in \cite{CCMNT,M2015} that $K_n$ satisfies an upper bound which only depends on $n$ and on the choice of the measure~\eqref{eq:jacobi_measure} through the values of $\theta_1$ and $\theta_2$.  

\begin{lemma}
\label{thm:bounds_K_lemma}
Let $d\mu$ be the measure defined in~\eqref{eq:jacobi_measure}. 
Then it holds
\begin{equation}
\label{eq:up_bounds_cheby_jacobi}
K_n \leq 
\left\{
\begin{array}{ll}
n^{\frac{ \ln 3 }{\ln 2}},
&
\quad 
\textrm{ if } \  \theta_1=\theta_2=-\frac12, 
\\
n^{2\max\{\theta_1,\theta_2\}+2},
&
\quad \textrm{ if } \  \theta_1,\theta_2\in\mathbb{N}_0.
\end{array}
\right.
\end{equation}
\end{lemma}

A remarkable property of both algebraic upper bounds in~\eqref{eq:up_bounds_cheby_jacobi} is that the exponent of $n$ is independent of the dimension $d$, and of the shape of the downward closed set $\Lambda$. 
Both upper bounds are sharp in the sense that equality holds for multi-index sets of rectangular type $\Lambda=R_\nu$ corresponding to tensor product polynomial spaces.

As an immediate consequence of Theorem~\ref{thm:theo1} and Lemma~\ref{thm:bounds_K_lemma}, we have the next corollary. 
\begin{corollary}
\label{thm:corollary_std_ls}
For any $r>0$, with multivariate polynomial spaces $\mathbb{P}_\Lambda$ and $\Lambda$ downward closed, if $m$ and $n$ satisfy 
\begin{equation}
\label{eq:bounds_corollary_std_ls}
\dfrac{m}{\ln m} \geq \kappa
\left\{
\begin{array}{ll}
n^{\frac{\ln 3}{\ln 2}},
&
\quad \textrm{ if } \ \theta_1=\theta_2=-\frac12,
\\
n^{2\max\{\theta_1,\theta_2 \} +2},
& \quad \textrm{ if } \ \theta_1,\theta_2\in\mathbb{N}_0,
\end{array}
\right.
\end{equation}
with $\kappa=\kappa(r)$ as in~\eqref{eq:condm}, then the same conclusions of Theorem~\ref{thm:theo1} hold true. 
\end{corollary}

Other types of results on the accuracy of least squares have been recently established in \cite{CMN}, under conditions of the same type as~\eqref{eq:bounds_corollary_std_ls}.

In some situations, for example when $n$ is very large, the conditions~\eqref{eq:bounds_corollary_std_ls} might require a prohibitive number of observations $m$. It is therefore a legitimate question to ask whether there exist alternative approaches with less demanding conditions than~\eqref{eq:bounds_corollary_std_ls} between $m$ and $n$. At best, we would like that $m$ is of order only slightly larger than $n$, for example by a logarithmic factor. In addition, the above analysis does not apply to situations where the basis functions $\phi_k$ are unbounded, such as when using Hermite polynomials in the expansion~\eqref{eq:hermite}. It is thus desirable to ask for the development of approaches that also cover this case.

These questions have an affirmative answer by considering {\it weighted least-squares methods}, as proposed in \cite{DH,JNZ,CM2016}. In the following, we survey some results from \cite{CM2016}. For the space $\mathbb{V}_n=V\otimes \mathbb{Y}_n$, the weighted least-squares approximation is defined as  
$$
u_W:=\argmin_{\widetilde u \in \mathbb{V}_n} \frac 1 m \sum_{i=1}^m w^i \| \widetilde u(y^i)-u^i \|_V^2,
$$
for some given choice of weights $w^i \geq 0$. This estimator is again computed by solving a linear system of normal equations now with the matrix  ${\bf G}$ with entries
$$
{\bf G}_{j,k}=\frac 1 m\sum_{i=1}^m w(y^i) \phi_j(y^i)\phi_k(y^i).
$$
Of particular interest to us are weights of the form
$$
w^i=w(y^i),
$$
where $w$ is some nonnegative function defined on $U$ such that 
\begin{equation}
\int_{U} w^{-1} \, d\mu=1.
\label{eq:cond_weight_w_ls}
\end{equation}  
We then denote by $d\sigma$ the probability measure 
\begin{equation}
d\sigma:= w^{-1} d\mu,
\label{eq:aux_measure_w_ls}
\end{equation}
and we draw the independent points $y^1,\ldots,y^m$ from $d\sigma$. The case $w \equiv 1$ and $d\sigma=d\mu$ corresponds to the previously discussed standard (unweighted) least-squares estimator $u_L$. As previously done for $u_L$, we associate to $u_W$ a truncated estimator $u_T$ and a conditioned estimator $u_C$, by replacing $u_L$ with $u_W$ in the corresponding definitions. 

Let us introduce the function
$$
y \mapsto k_{n,w}(y):=\sum_{j=1}^n w(y) |\phi_j(y)|^2,
$$
where once again $\{\phi_1,\dots,\phi_n\}$ is an arbitrary $L^2(U,d\mu)$ orthonormal basis of the space $\mathbb{Y}_n$. 
Likewise, we define
$$
K_{n,w}:=\|k_{n,w} \|_{L^\infty(U)}.
$$
The following result, established in \cite{CM2016} for real-valued functions, extends Theorem~\ref{thm:theo1} to this setting.
Its proof in the $V$-valued setting is exactly similar.

\begin{theorem}
\label{thm:theo10}
For any $r>0$, if $m$ and $n$ satisfy 
$$
\dfrac{m}{\ln m} \geq  \kappa \, K_{n,w},\;\; {\rm with}\;\; \kappa:=\kappa(r)=\frac{1-\ln 2}{2+2r},
$$
then the same conclusions of Theorem~\ref{thm:theo1} hold true with $u_L$ replaced by $u_W$. 
\end{theorem}

If we now choose  
\begin{equation}
w(y)=\dfrac{ n  }{ \sum_{j=1}^n |\phi_j(y)|^2  }, 
\label{eq:opt_weights_w_ls}
\end{equation}
that satisfies condition~\eqref{eq:cond_weight_w_ls} 
by construction, 
then 
the measure defined in~\eqref{eq:aux_measure_w_ls} takes the form 
\begin{equation}
d\sigma= \dfrac{ \sum_{j=1}^n |\phi_j(y)|^2  }{ n  } d\mu. 
\label{eq:opt_measure_w_ls}
\end{equation}
The choice~\eqref{eq:opt_weights_w_ls} also gives 
$$
K_{n,w}=\|k_{n,w} \|_{L^\infty(U)}=n, 
$$
and leads to the next result, as a consequence of the previous theorem. 

\begin{theorem}
\label{thm:theo11}
For any $r>0$, if $m$ and $n$ satisfy 
\begin{equation}
\dfrac{m}{\ln m} \geq  \kappa \ n, \;\; {\rm with}\;\; \kappa:=\kappa(r)=\frac{1-\ln 2}{2+2r},
\label{eq:condm_w_opt}
\end{equation}
then the same conclusions of Theorem~\ref{thm:theo1} hold true with $u_L$ replaced by $u_W$, with $w$ given by~\eqref{eq:opt_weights_w_ls} and the weights taken as $w^i=w(y^i)$.
\end{theorem}
 
The above theorem ensures stability and accuracy of the weighted least-squares approximation, under the minimal condition that $m$ is linearly proportional to $n$, up to a logarithmic factor.  
Clearly this is an advantage of weighted least squares compared to standard least squares, since condition~\eqref{eq:condm} is more demanding than~\eqref{eq:condm_w_opt} in terms of the number of observations $m$. 

However, this advantage comes with some drawbacks that we now briefly recall, see \cite{CM2016} for an extensive description. 
In general~\eqref{eq:aux_measure_w_ls} and~\eqref{eq:opt_measure_w_ls} are not product measures, even if $d\mu$ is one.
Therefore, the first drawback of using weighted least squares concerns the efficient generation of independent samples from multivariate probability measures,  
whose computational cost could be prohibitively expensive, above all when the dimension $d$ is large.  
In some specific settings, for example downward closed polynomial spaces $\mathbb{Y}_n=\mathbb{P}_\Lambda$ with $\#(\Lambda)=n$, and when $d\mu$ is a product measure, this drawback can be overcome. 
We refer to  \cite{CM2016}, where efficient sampling algorithms have been proposed and analyzed. For any $m$ and any downward closed set $\Lambda$, these algorithms generate $m$ independent samples with proven bounds on the required computational cost. 
The dependence on the dimension $d$ and $m$ of these bounds is linear. 
For the general measure~\eqref{eq:aux_measure_w_ls} the efficient generation of the sample is a nontrivial task, and remains a drawback of such an approach. 

The second drawback concerns the use of weighted least squares in a hierarchical context, where we are given a nested sequence $\Lambda_1\subset \ldots \subset \Lambda_n$ of downward closed sets, instead of a single such set $\Lambda$. 
Since the measure~\eqref{eq:opt_measure_w_ls} depends on $n$, the sets $(\Lambda_n)_{n\geq 1}$ are associated to different measures $(d\sigma_n)_{n\geq 1}$. Hence, recycling  samples from the previous iterations of the adaptive algorithm is not as straighforward as in the case of standard least squares. 

As a final remark, let us stress that the above results of Theorem~\ref{thm:theo10} and Theorem~\ref{thm:theo11} hold for general approximation spaces $\mathbb{Y}_n$ other than polynomials.

\section{Adaptive Algorithms and Extensions}
\label{sec:5}

\subsection{Selection of Downward Closed Polynomial Spaces}
\label{sec:5.1}

The interpolation and least-squares methods discussed in Section~\ref{sec:3} and Section~\ref{sec:4} allow us to construct polynomial approximations in $\mathbb{V}_\Lambda=V \otimes \mathbb{P}_\Lambda$ of the map~\eqref{eq:solmap} from its pointwise evaluations, for some given downward closed set $\Lambda$. For these methods, we have given several convergence results in terms of error estimates either in $L^\infty(U,V)$ or $L^2(U,V,d\mu)$. In some cases, these estimates compare favorably with the error of best approximation $\min_{\widetilde u \in \mathbb{V}_\Lambda} \|u-\widetilde u\|$ measured in such norms.

A central issue which still needs to be addressed is the choice of the downward closed set $\Lambda$, so that this error of best approximation is well behaved, for a given map $u$. Ideally, for each given $n$, we would like to use the set 
$$
\Lambda_n=\argmin_{\Lambda\in {\cal D}_n}\min_{\widetilde u\in \mathbb{V}_{\Lambda}} \|u-\widetilde u\|,
$$
where ${\cal D}_n$ is the family of all downward closed sets $\Lambda$ of cardinality $n$. However such sets $\Lambda_n$ are not explicitely given to us, and in addition the resulting sequence $(\Lambda_n)_{n\geq 1}$ is generally not nested.

Concrete selection strategies aim to produce ``suboptimal yet good'' nested sequences $(\Lambda_n)_{n\geq 1}$ different from the above. Here, an important distinction should be made between {\it nonadaptive} and {\it adaptive} selection strategies.

In nonadaptive strategies, the selection of $\Lambda_n$ is made in an a-priori manner, based on some available information on the given problem.
The results from Section~\ref{sec:2.3} show that, for relevant instances of solution maps associated to parametric PDEs, there exist nested
sequences $(\Lambda_n)_{n\geq 1}$ of downward closed sets such that $\#(\Lambda_n)=n$ and $\min_{\widetilde u\in \mathbb{V}_{\Lambda_n}} \|u-\widetilde u\|$ decreases with a given convergence rate $n^{-s}$ as $n\to \infty$. In addition, these results provide constructive strategies for building the sets $\Lambda_n$, 
since these sets are defined as the indices associated to the $n$ largest $\widehat \kappa_\nu:=\max_{\widetilde \nu\geq \nu} \kappa_{\widetilde \nu}$ like in Theorem~\ref{thm:theoleg}, or directly to the $n$ largest $\kappa_\nu$ like in Theorem~\ref{thm:theotaylor} and Theorem~\ref{thm:theoherm}, and since the $\kappa_{\nu}$ are explicitely given numbers.

In the case where we build the polynomial approximation by interpolation, Theorem~\ref{thm:theointer}
shows that a good choice of $\Lambda_n$ is produced by taking $\kappa_\nu$ to be either $\pi(\nu)b^\nu$ or $\pi(\nu) \beta(\nu)^2 b^\nu$ where $b=(\rho_j^{-1})_{j\geq 1}$ is such that~\eqref{eq:uearho} holds. In the case where we build the polynomial approximation by least-squares methods, the various results from Section~\ref{sec:4} show that under suitable assumptions, the error is nearly as good as that of best approximation in $L^2(U,V, d\mu)$ with respect to the relevant probability measure.
In the affine case, Theorem~\ref{thm:theoleg} shows that a good choice of $\Lambda_n$ is produced by taking $\kappa_\nu$ to be $b^\nu\beta(\nu)$ where $b=(\rho_j^{-1})_{j\geq 1}$ is such that~\eqref{eq:uearho} holds. In the lognormal case Theorem~\ref{thm:theoherm} shows that a good choice of $\Lambda_n$ is produced by taking $\kappa_\nu$ to be given by~\eqref{eq:bnuherm} where $b=(\rho_j^{-1})_{j\geq 1}$ is such that~\eqref{eq:psicr1} holds.

Let us briefly discuss the complexity of identifying the downward closed set $\Lambda_n$ associated to the $n$ largest $\widehat \kappa_\nu$.  For this purpose, we introduce for any downward closed set $\Lambda$ its set of {\it neighbors} defined by
$$
N(\Lambda):=\{\nu\in {\cal F}\setminus \Lambda \mbox{ such that } \Lambda\cup\{\nu\} \mbox { is downward closed} \}.
$$
We may in principle define $\Lambda_n=\{\nu^1,\dots,\nu^n\}$ by the following induction.
\begin{itemize}
\item
Take $\nu^1=0_{\cal F}$ as the null multi-index.
\item
Given $\Lambda_k=\{\nu^1,\dots,\nu^k\}$, choose a $\nu^{k+1}$ maximizing $\widehat \kappa_\nu$ over $\nu\in  N(\Lambda_k)$.
\end{itemize}
In the finite-dimensional case $d<\infty$, we observe that $N(\Lambda_k)$ is contained in the union of $N(\Lambda_{k-1})$ with the set consisting of the indices
$$
 \nu^k+e_j, \quad j=1,\dots,d,
$$
where $e_j$ is the Kroenecker sequence with $1$ at position $j$. As a consequence, since the values of the $\widehat \kappa_\nu$ have already been computed for $\nu\in N(\Lambda_{k-1})$, the step $k$ of the induction requires at most $d$ evaluations of $\widehat \kappa_\nu$, and therefore the overall computation of $\Lambda_n$ requires at most $nd$ evaluations.

In the infinite-dimensional case $d=\infty$, the above procedure cannot be practically implemented, since the set of neighbors has infinite cardinality. This difficulty can be circumvented by introducing a priority order among the variables, as done in the next definitions.

\begin{definition}
\label{thm:defanchoredseq}
A monotone nonincreasing positive sequence $(c_\nu)_{\nu\in{\cal F}}$ is said to be {\em anchored} if and only if
$$
l\leq j \implies c_{e_j} \leq c_{e_l}.
$$
A finite downward closed set $\Lambda$ is said to be {\em anchored} if and only if
$$
e_j\in \Lambda \quad {\rm and}\quad l\leq j \quad \implies \quad e_l\in \Lambda,
$$
where $e_l$ and $e_j$ are the Kroenecker sequences with $1$ at position $l$ and $j$, respectively.
\end{definition}

Obviously, if $(c_\nu)_{\nu\in{\cal F}}$ is anchored, one of the sets $\Lambda_n$ corresponding to its $n$ largest values is anchored. 
It is also readily seen that all sequences $(\widehat \kappa_\nu)_{\nu\in{\cal F}}$ that are used in Theorems~\ref{thm:theotaylor}, \ref{thm:theoleg}, \ref{thm:theoherm} or \ref{thm:theointer} for the construction of $\Lambda_n$ are anchored, provided that the sequence $b=(\rho_j^{-1})_{j\geq 1}$
is monotone nonincreasing. This is always the case up to a rearrangement of the variables.
For any anchored set $\Lambda$, we introduce the set of its {\it anchored neighbors} defined by
\begin{equation}
\widetilde N(\Lambda):=\{\nu\in N(\Lambda)\; : \; 
\nu_j=0\;\; {\rm if}\;\; j >j(\Lambda)+1\},
\label{eq:redneighbors}
\end{equation}
where 
$$
j(\Lambda):=\max\{j\; : \; \nu_j>0\; \mbox{for some}\; \nu\in \Lambda\}.
$$
We may thus modify in the following way the above induction procedure.
\begin{itemize}
\item
Take $\nu^1=0_{\cal F}$ as the null multi-index.
\item
Given $\Lambda_k=\{\nu^1,\dots,\nu^k\}$, choose a $\nu^{k+1}$ maximizing
$\widehat \kappa_\nu$ over $\nu\in \widetilde N(\Lambda_k)$.
\end{itemize}
This procedure is now feasible in infinite dimension. At each step $k$ the number of active variables is limited by $j(\Lambda_{k}) \leq k-1$, and the total number of evaluations of $\widehat \kappa_\nu$ needed to construct $\Lambda_n$ does not exceed $1+2+\dots +(n-1)\leq n^2/2$.

In adaptive strategies the sets $\Lambda_n$ are not a-priori selected, but instead they are built in a recursive way, based on earlier computations. For instance, one uses the previous set $\Lambda_{n-1}$ and the computed polynomial approximation $u_{\Lambda_{n-1}}$
to construct $\Lambda_n$.  If we impose that the sets $\Lambda_n$ are nested,
this means that we should select an index $\nu^n\notin \Lambda_{n-1}$ such that
$$
\Lambda_n:=\Lambda_{n-1}\cup\{\nu^n\}.
$$
The choice of the new index $\nu^n$ is further limited to $N(\Lambda_{n-1})$ if we impose that the constructed sets $\Lambda_n$ 
are downward closed, or to $\widetilde N(\Lambda_{n-1})$ if we impose that these sets are anchored.

Adaptive methods are known to sometimes perform significantly better
than their nonadaptive counterpart. In the present context,
this is due to the fact that the a-priori choices of $\Lambda_n$
based on the sequences $\kappa_\nu$ may fail to be optimal.
In particular, the guaranteed rate $n^{-s}$ based on such choices
could be pessimistic, and better rates could be obtained by other choices.
However, convergence analysis of adaptive methods is usually
more delicate. We next give examples of possible adaptive strategies in the
interpolation and least-squares frameworks.

\subsection{Adaptive Selection for Interpolation}
\label{sec:5.2}

We first consider polynomial approximations obtained by interpolation
as discussed in Section~\ref{sec:3}. The hierarchical form 
\begin{equation}
I_\Lambda u= \sum_{\nu\in \Lambda}\alpha_\nu B_\nu,
\label{eq:hierexp}
\end{equation}
may formally be viewed as a truncation of the 
expansion of $u$ in the hierarchical basis
$$
\sum_{\nu\in {\cal F}}\alpha_\nu B_\nu,
$$
which however may not always be converging, in contrast to 
the series discussed in Section~\ref{sec:2}.
Nevertheless, we could in principle take the same view, and 
use for $\Lambda_n$ the set of indices corresponding to the $n$ largest terms of~\eqref{eq:hierexp} measured in some given metric 
$L^p(U,V,d\mu)$. This amounts in choosing the indices of the
$n$ largest $w_\nu \|\alpha_\nu\|_V$, where 
the weight $w_\nu$ is given by
$$
w_\nu:=\|B_\nu\|_{L^p(U,d\mu)}.
$$
This weight is easily computable when $d\mu$ is a tensor product measure,
such as the uniform measure. In the case where $p=\infty$ and if we use the Leja sequence,
we know that $\|B_\nu\|_{L^\infty(U)}=1$ and therefore this amounts to choosing
the largest $\|\alpha_\nu\|_V$.

This selection strategy is not practially feasible
since we cannot afford this exhaustive search over ${\cal F}$.
However, it naturally suggests the following adaptive greedy
algorithm, which has been proposed in \cite{CCS1}.
\begin{itemize}
\item
Initialize $\Lambda_1:=\{0_{{\cal F}}\}$ with the null multi-index.
\item
Assuming that $\Lambda_{n-1}$ has been selected and that
the $(\alpha_\nu)_{\nu\in \Lambda_{n-1}}$ have been computed,
compute the $\alpha_\nu$ for $\nu\in N(\Lambda_{n-1})$. 
\item  
Set
\begin{equation}
\nu^n:=\argmax \{ w_\nu\|\alpha_\nu\|_V\; : \; \nu\in N(\Lambda_{n-1})\}.
\label{eq:greedy}
\end{equation}
\item
Define $\Lambda_{n}:=\Lambda_{n-1} \cup \{\nu^n\}$.
\end{itemize}

In the case where $p=\infty$ and if we use the Leja sequence, this strategy 
amounts in picking the index $\nu^n$ 
that maximizes the interpolation error $\|u(y_{\nu}) - I_{\Lambda_{n-1}}u(y_{\nu})\|_V$
among all $\nu$ in $N(\Lambda_{n-1})$. By the same considerations as previously discussed for the a-priori
selection of $\Lambda_n$, we find that in the finite-dimensional case, 
the above greedy algorithm requires at most $dn$ evaluation after $n$ steps. 
When working with infinitely many variables $(y_j)_{j\geq 1}$, we replace the infinite set $N(\Lambda_n)$ in the algorithm by the finite 
set of anchored neighbors $\widetilde N(\Lambda_n)$ defined by~\eqref{eq:redneighbors}. 
Running $n$ steps of the resulting greedy algorithm requires at most $n^2/2$ evaluations.

\begin{remark}
A very similar algorithm has been proposed in \cite{GG}
in the different context of adaptive quadratures, that is,
for approximating the integral of $u$ over the
domain $U$ rather than $u$ itself. In that case, the natural choice
is to pick the new index $\nu^n$ that maximizes
$|\int_U \Delta_\nu u \,  d\mu|$ over $N(\Lambda_n)$ or $\widetilde N(\Lambda_n)$.
\end{remark}

The main defect of the above greedy algorithm is that it may fail to converge, even
if there exist sequences $(\Lambda_n)_{n\geq 1}$ such that
$I_{\Lambda_n} u$ converges toward $u$. Indeed, if $\Delta_\nu u=0$ for a certain $\nu$, 
then no index $\widetilde \nu\geq \nu$ will
ever be selected by the algorithm. As an example,
if $u$ is of the form
$$
u(y)=u_1(y_1)u_2(y_2),
$$
where $u_1$ and $u_2$ are nonpolynomial smooth functions 
such that $u_2(t_0)=u_2(t_1)$, then the algorithm could
select sets $\Lambda_n$ with indices $\nu=(k,0)$ for $k=0,\dots,n-1$, since the
interpolation error at the point  $(t_k,t_1)$ vanishes. 

One way to avoid this problem is to adopt a more conservative selection
rule which ensures that all of ${\cal F}$ is explored, by alternatively using the rule~\eqref{eq:greedy}, or picking
the multi-index $\nu\in\widetilde N(\Lambda_n)$ which has appeared at the earliest stage
in the neighbors of the previous sets $\Lambda_k$. This is summarized by the following algorithm.

\begin{itemize}
\item
Initialize $\Lambda_1:=\{0_{\cal F}\}$ with the null multi-index.
\item
Assuming that $\Lambda_{n-1}$ has been selected and that
the $(\alpha_\nu)_{\nu\in \Lambda_{n-1}}$ have been computed,
compute the $\alpha_\nu$ for $\nu\in \widetilde N(\Lambda_{n-1})$. 
\item
If $n$ is even, set
\begin{equation}
\nu^n:=\argmax \{ w_\nu \|\alpha_\nu\|_V\; : \; \nu\in \widetilde N(\Lambda_{n-1})\}.
\label{eq:greedy_even}
\end{equation}
\item
If $n$ is odd, set
$$
\nu^{n}:={\rm argmin} \{k(\nu)\; : \; \nu\in  \widetilde N(\Lambda_{n-1})\}, \quad k(\nu):=\min\{k\; : \; \nu\in \widetilde N(\Lambda_k)\}.
$$
\item
Define $\Lambda_{n}:=\Lambda_{n-1} \cup \{\nu^n\}$.
\end{itemize}

Even with such modifications, the convergence of the interpolation error produced by this algorithm is not generally guaranteed. Understanding which additional assumptions on $u$ ensure convergence at some given rate, for a given univariate sequence $T$ such as Leja points, is an open problem.

\begin{remark}
Another variant to the above algorithms consists in choosing at the iteration $k$ more
than one new index at a time within  $N(\Lambda_{k-1})$ or $ \widetilde N(\Lambda_{k-1})$.
In this case, we have $n_k:=\#(\Lambda_k)\geq k$. For example we may choose the
smallest subset of indices that retains a fixed portion of the quantity
$\sum_{\nu \in \Lambda_{k-1}} w_\nu\|\alpha_\nu\|_V$. This type of modification 
turns out to be particularly relevant in the least-squares setting discussed in the next section.
\end{remark}

\subsection{Adaptive Selection for Least Squares}
\label{sec:5.3}

In this section we describe adaptive selections in polynomial spaces, for the least-squares methods that have been discussed in Section~\ref{sec:4}. 
We focus on adaptive selection algorithms based on the standard (unweighted) least-squares method. 

As a preliminary observation, it turns out that the most efficient available algorithms for adaptive selection of multi-indices might require the selection of more than one index at a time. Therefore, we adopt the notation that $n_k:=\#(\Lambda_k)\geq k$, where the index $k$ denotes the iteration in the adaptive algorithm. 

As discussed in Section~\ref{sec:4}, stability and accuracy of the least-squares approximation is ensured under suitable conditions between the number of samples and the dimension of the approximation space, see \emph{e.g.}~condition~\eqref{eq:bounds_corollary_std_ls}. 
Hence, in the development of reliable iterative algorithms,  
such conditions need to be satisfied at each iteration. When $d\mu$ is the measure~\eqref{eq:jacobi_measure} with shape parameters $\theta_1,\theta_2$, condition~\eqref{eq:bounds_corollary_std_ls} takes the form of 
\begin{equation}
\label{eq:condm_adaptive_ls}
\dfrac{m_k}{\ln m_k} \geq \kappa \
 n_k^{s}, 
\end{equation}
where $m_k$ denotes the number of samples at iteration $k$, 
and 
$$
s=
\left\{
\begin{array}{ll}
\ln 3 /\ln 2, 
&  
\quad
\textrm{ if } \ \theta_1=\theta_2=-\frac12, 
\\
{2\max\{\theta_1,\theta_2 \} +2}, 
&
\quad 
\textrm{ if } \ \theta_1,\theta_2\in\mathbb{N}_0. 
\end{array}
\right.
$$
Since $n_k$ increases with $k$, the minimal number of samples $m_k$ 
that satisfies~\eqref{eq:condm_adaptive_ls} 
has to increase as well at each iteration. 
At this point, many different strategies can be envisaged for progressively increasing  $m_k$ such that~\eqref{eq:condm_adaptive_ls} remains satisfied at each iteration $k$. 
For example, one can double the number of samples by choosing $m_{k}=2 m_{k-1}$ whenever \eqref{eq:condm_adaptive_ls} is broken, and keep $m_k=m_{k-1}$ otherwise. 
The sole prescription for applying Corollary~\ref{thm:corollary_std_ls} is that 
the samples are independent and drawn from $d\mu$. 
Since all the samples at all iterations are drawn from the same measure $d\mu$, at the $k$th iteration, 
where $m_k$ samples are needed, it is possible to use $m_{k-1}$ samples from the previous iterations, thus generating only $m_{k}-m_{k-1}$ new samples.

We may now present a first adaptive algorithm based on standard least squares.  
\begin{itemize}
\item 
Initialize $\Lambda_1:=\{0_{\cal F}\}$ with the null multi-index.
\item 
Assuming that $\Lambda_{k-1}$ has been selected, 
compute the least-squares approximation 
$$u_L=\sum_{\nu \in \Lambda_{k-1} \cup N(\Lambda_{k-1}) } c_\nu \phi_\nu $$
of $u$ in $\mathbb{V}_{\Lambda_{k-1}\cup N(\Lambda_{k-1})}$, 
using a number of samples $m_k$ that satisfies condition~\eqref{eq:condm_adaptive_ls} with $n_k=  
\#(\Lambda_{k-1}\cup N(\Lambda_{k-1}))$. 
\item Set
\begin{equation}
\nu^k:=\argmax_{ \nu\in N(\Lambda_{k-1})} | c_\nu |^2.
\label{eq:greedy_ls}
\end{equation}
\item 
Define $\Lambda_{k}:=\Lambda_{k-1} \cup \{\nu^k\}$.
\end{itemize}

Similarly to the previously discussed interpolation algorithms,
in the case of infinitely many variables $(y_j)_{j\geq 1}$ the set $N(\Lambda_k)$ is infinite and should be replaced by the finite set of anchored neighbors $\widetilde N(\Lambda_k)$
defined by~\eqref{eq:redneighbors}.
As for interpolation, we may define a more conservative 
version of this algorithm in order to ensure that 
all of ${\cal F}$ is explored. For example, when $k$ is even, we define 
$\nu^{k}$ according to~\eqref{eq:greedy_ls}, and when $k$ is odd we pick
for $\nu^k$ the multi-index $\nu\in\widetilde N(\Lambda_k)$ which has appeared at the earliest stage
in the neighbors of the previous sets $\Lambda_k$. 
The resulting algorithm is very similar to the one presented for interpolation, with obvious modifications due to the use of least squares. 

As announced at the beginning, it can be advantageous to select more than one index at a time from $\widetilde{N}(\Lambda_{k-1})$, at each iteration $k$ 
of the adaptive algorithm. 
For describing the multiple selection of indices from $\widetilde{N}(\Lambda_{k-1})$, we introduce the so-called {\it bulk chasing} procedure.   
Given a finite set 
$R \subseteq \widetilde{N}(\Lambda_{k-1})$, 
a nonnegative function ${\cal E}:R\to\mathbb{R}$ and a parameter $\alpha \in (0,1]$, we define the procedure $\textrm{bulk}:= \textrm{bulk}(R,{\cal E},\alpha)$ that  
computes a set 
$F\subseteq R$
of minimal positive cardinality 
such that 
$$
\sum_{\nu \in F} {\cal E}(\nu) \geq \alpha \sum_{\nu \in R} {\cal E}(\nu). 
$$
A possible choice for the function ${\cal E}$  
is 
$$
{\cal E}(\nu)={\cal E}_L(\nu):=|c_\nu|^2, \quad \nu \in R, 
$$ 
where $c_\nu$ is given from an available least-squares estimator 
$$
u_L=\sum_{\nu \in \Lambda} c_\nu \phi_\nu, 
$$
that has been already computed on any downward closed set $R \subset \Lambda\subseteq \Lambda_{k-1}\cup \widetilde{N}(\Lambda_{k-1})$. 
Another choice for ${\cal E}$ is  
$$
{\cal E}(\nu)={\cal E}_M( \nu):= 
\langle \phi_{ \nu}, 
u - \widetilde{u}_L
 \rangle_{m_{k-1}}, 
\quad \nu \in R,  
$$
where $\widetilde{u}_L$ is the truncation to $\Lambda_{k-1}$ of a least-squares estimator $u_L= \sum_{ \nu  \in \Lambda} c_{\nu} \phi_{\nu}$ that has been already computed on any downward closed set $\Lambda_{k-1} \subset \Lambda\subseteq  \Lambda_{k-1} \cup \widetilde{N}(\Lambda_{k-1})$, using a number of samples $m_{k-1}$ that satisfies condition~\eqref{eq:condm_adaptive_ls} with $n_k=\#(\Lambda)$. The discrete norm in ${\cal E}_M(\nu)$ uses the same $m_{k-1}$ evaluations of $u$ that have been used to compute the least-squares approximation $u_L$ on $\Lambda$.  

Both ${\cal E}_L(\nu)$ and ${\cal E}_M(\nu)$ should be viewed as estimators of the coefficient $\langle u, \phi_\nu \rangle$. 
The estimator ${\cal E}_M(\nu)$ is of Monte Carlo type and computationally cheap to calculate.  
Combined use of the two estimators leads to the next algorithm for greedy selection with bulk chasing, that has been proposed in \cite{M2013}.

\begin{itemize}
\item 
Initialize $\Lambda_1:=\{0_{\cal F}\}$ with the null multi-index, and choose $\alpha_1,\alpha_2\in(0,1]$. 
\item 
Assuming that $\Lambda_{k-1}$ has been selected, set
\begin{equation}
F_1=\textrm{bulk}(\widetilde{N}(\Lambda_{k-1}),
{\cal E}_M
,\alpha_1 ),
\label{eq:greedy_ls_dorfler1}
\end{equation}
where ${\cal E}_M$ uses the least-squares approximation $u_L=\sum_{\nu \in \Lambda} c_\nu \phi_\nu $ of $u$ in $\mathbb{V}_{\Lambda}$ that has been calculated at iteration $k-1$ on a downward closed set $ \Lambda_{k-1} \subset \Lambda\subseteq \Lambda_{k-1} \cup \widetilde{N}(\Lambda_{k-1})$ using a number of samples $m_{k-1}$ that satisfies~\eqref{eq:condm_adaptive_ls} with $n_k=  \#(\Lambda)$.

\item 
Compute the least-squares approximation 
\begin{equation}
u_L=\sum_{\nu \in \Lambda_{k-1} \cup F_1} c_\nu \phi_\nu 
\label{eq:computation_LS_greedy}
\end{equation}
of $u$ on $\mathbb{V}_{\Lambda_{k-1} \cup F_1}$ using a number of samples $m_{k}$ that satisfies~\eqref{eq:condm_adaptive_ls} with $n_k=\#(\Lambda_{k-1}\cup F_1)$. 

\item 
Set  
\begin{equation}
F_2=\textrm{bulk}(F_1,
{\cal E}_L
,\alpha_2 ), 
\label{eq:greedy_ls_dorfler2}
\end{equation}
where ${\cal E}_L$ uses the least-squares approximation $u_L$ computed on $\Lambda_{k-1}\cup F_1$. 
\item Define $\Lambda_{k}=\Lambda_{k-1} \cup F_2$.
\end{itemize}

The set $\widetilde{N}(\Lambda_{k-1})$ can be large, and might contain many indices that are associated to small coefficients. 
Discarding these indices is important in order to avoid unnecessary computational burden in the calculation of the least-squares approximation.  
The purpose of the bulk procedure~\eqref{eq:greedy_ls_dorfler1} is to
perform a preliminary selection of a set $F_1 \subseteq \widetilde{N}(\Lambda_{k-1})$ of indices, using the cheap estimator ${\cal E}_M$. At iteration $k$, ${\cal E}_M$ in \eqref{eq:greedy_ls_dorfler1} uses the estimator computed in~\eqref{eq:computation_LS_greedy} at iteration $k-1$ and truncated to $\Lambda_{k-1}$. 
Afterwards, at iteration $k$, the least-squares approximation in~\eqref{eq:computation_LS_greedy} is calculated on $\Lambda_{k-1}\cup F_1$, 
using a number of samples $m_k$ which satisfies condition~\eqref{eq:condm_adaptive_ls}, 
with $n_k=\#(\Lambda_{k-1}\cup F_1)$. 
The second bulk procedure~\eqref{eq:greedy_ls_dorfler2} selects 
a set $F_2$ of indices from $F_1$, using the more accurate estimator ${\cal E}_L$. 
The convergence rate of the adaptive algorithm depends on the values given to the parameters $\alpha_1$ and $\alpha_2$.

Finally we mention some open issues related to the development of adaptive algorithms 
using the weighted least-squares methods discussed in Section~\ref{sec:4}, instead of standard least squares. 
In principle the same algorithms described above   
can be used with the weighted least-squares estimator $u_W$ replacing the standard least-squares estimator $u_L$,
provided that, at each iteration $k$, the number of samples $m_k$ satisfies 
$$
\dfrac{m_k}{\ln m_k} \geq \kappa \
 n_k, 
$$
and that the samples are drawn from the optimal measure, see Theorem~\ref{thm:theo11}.  
This ensures that at each iteration $k$ of the adaptive algorithm, the weighted least-squares approximation 
remains stable and accurate. However, no guarantees on stability and accuracy are ensured if the
above conditions are not met, for example when the samples from previous iterations are recycled.

\subsection{Approximation in Downward Closed Spaces: beyond Polynomials}
\label{sec:5.4}

The concept of downward closed approximation spaces can be generalized beyond the polynomial setting. We start from a countable index set $S$ equipped with a partial order $\leq$, 
and assume that there exists a root index $0_S\in S$ such that $0_S \leq \sigma$ for all $\sigma\in S$.
We assume that $(B_\sigma)_{\sigma\in S}$ is a basis of functions defined on $[-1,1]$ such that $B_{0_S} \equiv 1$. We then define by tensorization a basis of functions
on $U=[-1,1]^d$ when $d<\infty$, or $U=[-1,1]^\N$ in the case of infinitely many variables, according to
$$
B_\nu(y)=\prod_{j\geq 1} B_{\nu_j} (y_j), \quad \nu:=(\nu_j)_{j\geq 1}\in {\cal F},
$$
where ${\cal F}:=S^d$ in the case $d<\infty$, or ${\cal F}=\ell^0(\N,S)$, \emph{i.e.}~the set of finitely supported sequences, in the case $d=\mathbb{N}$. 

The set ${\cal F}$ is equipped with a partial order induced by its univariate counterpart: $\nu\leq \widetilde \nu$ if and only if $\nu_j\leq \widetilde \nu_j$ for all $j\geq 1$. We may then define downward closed sets $\Lambda\subset {\cal F}$ in the same way as in Definition~\ref{thm:defdc} which corresponds to the particular case $S=\mathbb{N}$. 
We then define the associated downward closed approximation space by
$$
\mathbb{V}_\Lambda:=V\otimes \mathbb{B}_\Lambda, \quad \mathbb{B}_\Lambda:={\rm span}\{B_\nu\, : \, \nu\in \Lambda\},
$$
that is the space of functions of the form $\sum_{\nu\in \Lambda} u_\nu B_\nu$ with $u_\nu\in V$.

Given a sequence $T=(t_\sigma)_{\sigma\in S}$ of pairwise distinct points we say that the basis $(B_\sigma)_{\sigma\in S}$ is hierarchical when it satisfies
$$
B_\sigma(t_\sigma)=1 \ {\rm and} \  B_\sigma(t_{\widetilde\sigma})=0 \ {\rm if} \ \widetilde \sigma \leq \sigma \ {\rm and} \ \widetilde \sigma \neq \sigma.
$$
We also define the tensorized grid 
$$
y_\nu:=(t_{\nu_j})_{j\geq 1}\in U.
$$
Then, if $\Lambda\subset {\cal F}$ is a downward closed set, we may define an interpolation operator $I_\Lambda$ onto $V_\Lambda$ associated to the grid 
$$
\Gamma_\Lambda:=\{y_\nu \; : \; \nu\in \Lambda\}.
$$
In a similar manner as in the polynomial case, this operator is defined inductively by
$$
I_\Lambda u:=I_{\widetilde \Lambda} u+\alpha_{\nu} B_\nu,\quad \alpha_\nu:=\alpha_\nu(u)=u(y_\nu)-I_{\widetilde \Lambda} u(y_\nu),
$$
where $\nu\notin\widetilde \Lambda$ and $\widetilde\Lambda$ is any downward closed set such that $\Lambda=\widetilde\Lambda\cup\{\nu\}$. We initialize this computation with $\Lambda_1=\{0_{\cal F}\}$, where $0_{\cal F}$ is the null multi-index, by defining $I_{\Lambda_1}u$ as the constant function with value $u(y_{0_{\cal F}})$. 

Examples of relevant hierarchical systems include the classical piecewise linear hierarchical basis functions.
In this case the set $S$ is defined by
$$
S = \{\lambda_{-1},\lambda_1,(0,0)\} \cup \left\{(j,k)\; : \;  -2^{j-1}\leq k\leq 2^{j-1}-1,\; j=1,2,\dots \right\}
$$
equipped with the partial order $\lambda_{-1} \leq \lambda_{1} \leq (0,0)$ and 
$$
(j,k)\leq (j+1,2k), \quad\quad
(j,k)\leq (j+1,2k+1),  
\quad \quad (j,k) \in S.
$$
The set $S$ is thus a binary tree where $\lambda_{-1}$ is the root node, $(0,0)$ is a child 
of $\lambda_1$ which is itself a child of $\lambda_{-1}$, every node $(j,k)$ has two children
$(j+1,2k)$ and $(j+1,2k+1)$, and the relation $\widetilde \lambda \leq \lambda$ means that $\widetilde \lambda$ is a parent of $\lambda$.
The index $j$ corresponds to the level of refinement, \emph{i.e.}~the depth of the node in the binary tree. 
	
We associate with $S$ the sequence
$$
T :=  \{t_{\lambda_{-1}},t_{\lambda_1},t_{(0,0)}\} \cup \left\{t_{(j,k)} := \frac {2k+1}{2^j}: (j,k)\in S, j\geq1 \right\} , 
$$
where $t_{\lambda_{-1}}=-1$,  $t_{\lambda_1}=1$ and $t_{(0,0)}=0$. The 
hierarchical basis of piecewise linear functions defined over 
$[-1,1]$ is then given by
$$
B_{\lambda_{-1}} \equiv 1,\quad 
B_{\lambda_{1}} (t) = \frac {1+t}{2},\quad
B_{(j,k)} (t) = H(2^j(t-t_{(j,k)})),\quad 
(j,k) \in S,
$$
where
$$
{H}(t):=\max\{0,1-|t|\},
$$
is the usual hat function. In dimension $d=1$, the hierarchical interpolation amounts in the following steps: start by approximating $f$ with the constant function equal to $f(-1)$, then with the affine function that coincides with $f$ at $-1$ and $1$, then with the piecewise affine function that coincides with $f$ at $-1$, $0$ and $-1$; afterwards refine the approximation in further steps by interpolating $f$ at the midpoint of an interval between two adjacents interpolation points. 

Other relevant examples include piecewise polynomials, hierarchical basis functions, and more general interpolatory wavelets, see \cite{Co} for a survey.

\end{document}